\documentclass[a4paper,10pt]{amsart}

\usepackage[utf8]{inputenc} 
\usepackage[T1]{fontenc}
\usepackage{mathtools,amssymb,amsthm} 
\usepackage[all]{xy}
\usepackage{todonotes}
\usepackage{xparse} 
\usepackage{hyperref}

\usepackage{chngcntr}

\swapnumbers 
\theoremstyle{definition} 
\newtheorem{Def}{Definition}[subsection] 
\theoremstyle{plain} 
\newtheorem{Pro}[Def]{Proposition} 
\newtheorem{Lem}[Def]{Lemma} 
\newtheorem{The}[Def]{Theorem} 
\newtheorem{Cor}[Def]{Corollary} 

\counterwithin*{equation}{section}

\theoremstyle{remark} 

\newtheorem{Rem}[Def]{Remark} 
\newtheorem{Exa}[Def]{Example} 

\DeclareMathOperator{\Gf}{\mathbf{G}}
\DeclareMathOperator{\Gfd}{\mathbf{G}^{*}}
\DeclareMathOperator{\Irr}{Irr}
\DeclareMathOperator{\Lf}{\mathbf{L}}
\DeclareMathOperator{\Sf}{\mathbf{S}}
\DeclareMathOperator{\Tf}{\mathbf{T}}
\DeclareMathOperator{\Pf}{\mathbf{P}}

\newcommand{\fr}{\mathsf{F}}
\newcommand{\frd}{\mathsf{F}^*}

\newcommand{\dl}[2]{\mathcal{E}(#1,#2)}
\newcommand{\dll}[2]{\mathcal{E}_{\ell}(#1,#2)}

\newcommand{\Aa}{\mathbf{A}}
\newcommand{\Bb}{\mathbf{B}}
\newcommand{\Cc}{\mathbf{C}}
\newcommand{\Dd}{\mathbf{D}}
\newcommand{\Ee}{\mathbf{E}}
\newcommand{\Ff}{\mathbf{F}}
\newcommand{\Gg}{\mathbf{G}}

\DeclareMathOperator{\rank}{rank}
\DeclareMathOperator{\defect}{defect}

\DeclareMathOperator{\Rep}{Rep}
\NewDocumentCommand{\rep}{ O{} O{} m }{\Rep_{#1}^{#2}(#3)}
\newcommand{\TG}{\mathcal{T}(G)}

\newcommand{\Ql}{\ladic{Q}}
\newcommand{\Zl}{\ladic{Z}}

\DeclareMathOperator{\BT}{BT}

\newcommand{\simu}{\sim^1}

\def\NM{{\mathbb{N}}}

\def\QM{{\mathbb{Q}}}
\def\FM{{\mathbb{F}}}
\def\ZM{{\mathbb{Z}}}

\def\CM{{\mathbb{C}}}

\def\BG{{\mathfrak B}}

\def\tG{{\mathfrak t}}

\def\ZG{{\mathfrak Z}}

\def\HG{{\mathfrak H}}

\def\TG{{\mathfrak T}}

\def\HC{{\mathcal H}}

\def\OC{{\mathcal O}}

\def\LC{{\mathcal L}}
\def\EC{{\mathcal E}}

\def\GC{{\mathcal G}}

\def\BC{{\mathcal B}}

\def\tr{{\rm tr}}

\def\simto{\buildrel\hbox{\tiny{$\sim$}}\over\longrightarrow}

\def\leq{\leqslant}
\def\geq{\geqslant}
\def\injo{\hookrightarrow}

\def\o#1{\overline{#1}}

\def\To#1{\buildrel\hbox{\tiny{$#1$}}\over\longrightarrow}

\def\ker{{\rm ker}}

\def\Hom{\mathop{\hbox{\rm Hom}}\nolimits}

\def\Rep{{\rm {R}ep}}

\def\dim{\mathop{\mbox{\rm dim}}\nolimits}
\def\dim{{\rm dim}}

\def\BT{{\BC}}

\def\Ql{\QM_{\ell}}
\def\Zl{\ZM_{\ell}}

\def\Fr{{\rm Fr}}

\def\HG{{\hat G}}

 \def\HG{\widehat{\mathbf{G}}}
 \def\HT{\widehat{\mathbf{T}}}
 \def\HB{\widehat{\mathbf{B}}}

\providecommand\sslash{\mathbin{/\mkern-5.5mu/}}

\title{Depth zero representations over $\overline{\mathbb{Z}}[\frac 1p]$}
\author{Jean-François Dat}
\address{Jean-Fran\c cois Dat, Institut de Mathématiques de Jussieu, Sorbonne Université -- Université de
  Paris -- CNRS, 4 Place Jussieu, 75252 Paris cedex.}
\thanks{JFD acknowledges the support of the ANR  through the grant ``Coloss'' ANR-19-CE40-0015} 
\email{jean-francois.dat@imj-prg.fr}

\author{Thomas Lanard}
\address{Thomas Lanard, CNRS, Laboratoire de mathématiques de Versailles, Université Paris-Saclay, UVSQ, 78000, Versailles, France}
\email{thomas.lanard@uvsq.fr}

\begin{document}

\begin{abstract}
  We consider the category of depth $0$ representations of a $p$-adic
  quasi-split reductive group with coefficients in $\o\ZM[\frac 1p]$.
  We prove that the blocks of this category are in natural
  bijection with the connected components of the space of tamely ramified Langlands
  parameters for $G$ over $\o\ZM[\frac 1p]$. As a particular case, this depth $0$
  category is thus indecomposable when the group is tamely ramified. Along the way we prove a
  similar result for finite reductive groups. As an application, we deduce that the semi-simple
  local Langlands correspondence $\pi\mapsto \varphi_{\pi}$
  constructed by Fargues and Scholze takes depth $0$ representations
  to tamely ramified parameters, using a motivic version of their construction recently
  announced by Scholze. We also bound the restriction of
  $\varphi_{\pi}$ to tame inertia in terms of the Deligne-Lusztig parameter
  of $\pi$ and show, in particular, that $\varphi_{\pi}$ is unramified
  if $\pi$ is unipotent.
\end{abstract}

\maketitle

\section{Main results}

We prove two results on the representation theory of finite reductive
groups and on that of $p$-adic reductive groups. We first state these results and then
explain our motivations and some connections to the existing literature.

\begin{The}[Theorem \ref{primitive}]
  Let $\Gf$ be a reductive group over
  $\o\FM_{p}$ and $\fr$ the Frobenius map associated to a $\FM_{p^{r}}$-rational structure for some $r\geq 1$. Then
  the category $\Rep_{\o\ZM[\frac 1p]}(\Gf^{\fr})$ is indecomposable. Equivalently, the central idempotent $1$ in
  $\o\ZM[\frac 1p]\Gf^{\fr}$ is primitive.
\end{The}

This result initially appeared as one step in our study of the $p$-adic case below. We
have decided to single it out because the statement is simple and quite natural. It might
be an interesting problem to try and devise criteria for a similar statement to hold true for an
abstract finite group $G$ and a prime divisor $p$ of $|G|$.

\medskip

Let now $\Gf$ be  a reductive group over a local non-archimedean field $F$
with residue field $k_{F}:=\FM_{p^{r}}$. We put $G:=\Gf(F)$. For any
commutative ring $R$ in which $p$ is invertible, we denote by
$\Rep_{R}(G)$ the category of smooth $R G$-modules. The Bernstein center
$\ZG_{R}(G)$  is by definition the center of this category. We refer
to subsection~\ref{sec:depth-0-summand} for the definition of
\emph{depth $0$} smooth $R G$-modules. They form a direct factor
subcategory $\Rep_{R}^{0}(G)$, which corresponds to some idempotent
$\varepsilon_{0}\in \ZG_{R}(G)$.  The following statement is a sample
of what we prove about $\Rep_{\o\ZM[\frac 1p]}^{0}(G)$.

\begin{The}[Theorem \ref{quasisplittame}] \label{the_tame}
  Suppose that $\Gf$ is quasi-split and tamely ramified over
  $F$. Then the abelian category $\Rep_{\o\ZM[\frac 1p]}^{0}(G)$ is
  indecomposable. Equivalently,   $\varepsilon_{0}$
  is a primitive idempotent of $\ZG_{\o\ZM[\frac 1p]}(G)$.
\end{The}

This result was mainly inspired by its ``dual'' counterpart in
\cite{DHKM}, where the moduli space
$Z^{1}(W_{F}^{0},\HG)$ of Langlands parameters for $\mathbf{G}$ was constructed over
$\o\ZM[\frac 1p]$ and studied. Concretely, $Z^{1}(W_{F}^{0},\HG)$ classifies $1$-cocycles
$W_{F}^{0}\To{}\HG$ where :
\begin{itemize}
  \item  $\HG$ denotes the dual  group of
        $\mathbf{G}$, considered as a split pinned reductive group scheme over
        $\o\ZM[\frac 1p]$, and endowed with an action of the Galois group $\Gamma_{F}$ of $F$ that
        preserves the  pinning
  \item  $W_{F}^{0}\subset W_{F}\subset \Gamma_{F}$ is some
        modification of the Weil group of $F$.
\end{itemize}
In \cite{DHKM}, this moduli space is decomposed according to the restriction of $1$-cocycles
to the wild inertia subgroup $P_{F}\subset W_{F}^{0}$. In particular, the ``tame'' summand
$Z^{1}(W_{F}^{0},\HG)_{\rm tame}$ parametrizes $1$-cocycles whose restriction to $P_{F}$
is locally (for the \'etale topology) conjugate to the trivial cocycle. According to \cite[Thm 4.29]{DHKM}, this
summand is connected, provided that $\mathbf{G}$ is tamely ramified. Since tame
parameters are supposed to correspond to depth $0$ representations by any form of local
Langlands correspondence, we like to see the last theorem as the group side analogue of this
connectedness result on the parameter side.  Interestingly, the proof of \cite[Thm
  4.29]{DHKM} consists in, first, classifying the connected components of
$Z^{1}(W_{F}^{0},\HG)_{_{\o\ZM_{\ell}},\rm tame}$ for each $\ell\neq p$, and then, using
different $\ell$'s to get the result. Similarly, one way to formulate the
indecomposability of $\Rep^{0}_{\o\ZM[\frac 1p]}(G)$ is as follows (we refer to \ref{remblocksequences} for the notion of $\ell$-block used here).
\begin{Cor} Under the same hypothesis on $\mathbf{G}$,
  given $\pi,\pi'$ two irreducible $\o\QM G$-modules of depth $0$, there is a sequence of
  primes $\ell_{1},\cdots, \ell_{r}$ and a sequence
  $\pi_{0}=\pi,\pi_{1},\cdots,\pi_{r}=\pi'$ of irreducible $\o\QM G$-modules such that
  $\pi_{i-1}$ and $\pi_{i}$ belong to the same $\ell_{i}$-block for each $i=1,\cdots,r$.
\end{Cor}

We note that Sécherre and Stevens have used in \cite{SecherreStevensJL} a similar statement in the context of inner
forms of $GL(n)$ (but for arbitrary ``endoclasses'') in order to gain control on the
Jacquet-Langlands correspondence for \emph{complex} representations. This provides a
striking example of how to use this kind of results for problems a priori unrelated to
congruences. The application we give in \ref{sec:sample-potent-appl} below is actually in
the same vein.

\medskip

Meanwhile, let us observe that it is not always true, even for $\mathbf{G}$ a torus, that the tame summand of
$Z^{1}(W_{F}^{0},\HG)$ is connected. In Theorem \ref{thm_main_dual}, we work out the decomposition of
$Z^{1}(W_{F}^{0},\HG)_{\rm tame}$ into connected components for general $\mathbf{G}$, and
we prove in particular that these connected components are simply transitively permuted by a
certain abelian $p$-group of ``central cocycles''. On the other hand,
in Theorem \ref{thm_main}, we work out the decomposition of $\Rep^{0}_{\o\ZM[\frac 1p]}(G)$ into a
product of blocks for quasi-split $G$, and prove  that these blocks
are simply transitively permuted by
a certain abelian $p$-group of characters of $G$. After identifying
these two $p$-groups and matching the principal component with the principal block, we then conclude:

\begin{The}\label{the_main}
  Assume $\mathbf{G}$ is quasi-split over $F$. Then there is a natural bijection between
  connected components of $Z^{1}(W_{F}^{0},\HG)_{\rm tame}$ and blocks of $\Rep^{0}_{\o\ZM[\frac 1p]}(G)$.
\end{The}
Again, this implies that, under the same quasi-splitness hypothesis,
a $\pi\in\Irr_{\o\QM}(G)$ of depth $0$ can be connected to a
depth $0$ character of $p$-power order through a sequence of ``congruences'' modulo
different primes.

\medskip

For a non quasi-split group $\mathbf{G}$, there is in general an additional ``relevance'' condition on
$1$-cocycles for them to provide Langlands parameters of $G$. Although this relevance
condition might mess up with connected components, we believe it does not actually happen,
i.e. the above result should be true with no quasi-split assumption. As evidence for
this expectation, we prove:

\begin{The}[Corollary \ref{cornqs}]
  Suppose that $p$ does not divide $ |\pi_{1}(\Gf_{\rm der})|$ and the torus $\Gf_{\rm ab}:=\Gf/[\Gf,\Gf]$ is
  $P_{F}$-induced. Then $\Rep_{\o\ZM[\frac 1p]}^{0}(G)$ is
  indecomposable. 
\end{The}

This is in accordance with the fact that, if $p$ does not divide $|\pi_{0}(Z(\HG))|$ and
$(Z(\HG)^{\circ})^{ P_{F}}$ is connected, then $Z^{1}(W_{F}^{0},\HG)_{\rm tame}$ is connected.

\medskip

\subsection{Some applications to the Fargues-Scholze semisimple correspondence.} \label{sec:sample-potent-appl}
Among the major recent breakthroughs towards constructing the conjectural local Langlands
correspondence for a $p$-adic group $G$,   Fargues and Scholze \cite{FS}
have recently used new geometric tools to attach to any
irreducible representation $\pi$ of $G$, a semisimple local Langlands
parameter $\varphi_{\pi}$. Their construction is compatible with parabolic induction and
local class field theory,
so that, for example, the semisimple parameter attached to an
``unramified  principal series'' is indeed
unramified, as expected. However it is in general very difficult
to say anything on this parameter, especially when $\pi$ is
cuspidal. For example, it is not a priori clear  that
$\varphi_{\pi}$ is tamely ramified whenever $\pi$ has depth $0$. 
Fargues and Scholze's construction actually provides a map ${\rm FS_{\ell}}$
$$\OC(Z^{1}(W_{F}^{0},\HG))^{\HG}_{\o\ZM_{\ell}}  \buildrel\hbox{\tiny $\approx$}\over\longleftarrow
  {\rm \EC xc}(W_{F}^{0},\HG)_{\o\ZM_{\ell}}\To{{\rm FS_{\ell}}}\ZG_{\o\ZM_{\ell}}(G)$$
for each prime $\ell\neq p$. Here, ${\rm \EC xc}(W_{F}^{0},\HG)$ denotes the so-called
``excursion algebra'' over $\o\ZM[\frac 1p]$, and the symbol $\approx$
denotes a universal
homeomorphism to $\OC(Z^{1}(W_{F}^{0},\HG)^{\HG}$
(thus inducing a bijection on geometric points and on sets of connected components).
The $\HG(\o\QM_{\ell})$ conjugacy
class of semisimple parameters $\varphi_{\pi}$ associated to
$\pi\in{\rm Irr}_{\o\QM_{\ell}}(G)$ is then given by the  $\o\QM_{\ell}$-point of
$Z^{1}(W_{F}^{0},\HG)\sslash\HG$ obtained by composing ${\rm FS}_{\ell}$ with
the infinitesimal character $\ZG_{\o\ZM_{\ell}}(G)\To{}{\rm
    End}_{\o\QM_{\ell}}(\pi)=\o\QM_{\ell}$.
In particular, this
construction is compatible with congruences mod $\ell$.
Thanks to the description of the connected
components of $Z^{1}(W_{F}^{0},\HG))_{\o\ZM_{\ell}}$ in \cite[Thm 4.8]{DHKM}, this implies for
example that if $\pi, \pi'$ belong to the same block of $\Rep_{\o\ZM_{\ell}}(G)$, then the
restrictions of $\varphi_{\pi}$ and $\varphi_{\pi'}$ to the prime-to-$\ell$ inertia subgroup $I_{F}^{\ell}$ coincide.

In order to control $\varphi_{\pi}$ for a depth $0$ irreducible representation $\pi$, one
would like to use congruences modulo different primes, as encapsulated in our main results
above. This supposes to have some form of ``independence of $\ell$'' for the
Fargues-Scholze semisimple correspondence. Recently, Scholze \cite{Scholze_motiv} has
announced a solution to this problem that uses a version of his work with Fargues with
motivic coefficients. For our purposes, the upshot is that the family of maps ${\rm
      FS}_{\ell}$ for $\ell\neq p$ is induced by base change from a map $\rm FS_{mot}$ as follows
$$\OC(Z^{1}(W_{F}^{0},\HG))^{\HG}  \buildrel\hbox{\tiny $\approx$}\over\longleftarrow
  {\rm \EC xc}(W_{F}^{0},\HG) \To{{\rm FS_{mot}}}\ZG_{\o\ZM[\frac 1p]}(G).$$
The following is a consequence of Theorem \ref{the_main}.

\begin{Cor} \label{cor_FS}
  Take the existence of $\rm FS_{mot}$ for granted and assume that $\mathbf{G}$ is quasi-split. 
  Then the
  Fargues-Scholze parameter $\varphi_{\pi}$ of a depth $0$ irreducible representation
  $\pi$ is tamely ramified.
\end{Cor}

Let us spell out the argument in the case where $\mathbf{G}$ is  tamely
ramified, so that,  by Theorem \ref{the_tame}, the depth $0$ projector is actually a primitive idempotent in
$\ZG_{\o\ZM[\frac 1p]}(G)$.
Then the spectrum of $\varepsilon_{0}\ZG_{\o\ZM[\frac 1p]}(G)$ is mapped into a
single connected component of  $Z^{1}(W_{F}^{0},\HG)\sslash\HG$ under $\rm FS^{*}_{mot}$. In other words,
$\varphi_{\pi}$ has to be a geometric point of the same connected
component of $Z^{1}(W_{F}^{0},\HG)\sslash\HG$ as any other $\varphi_{\pi'}$ attached to a depth $0$
irreducible representation $\pi'$. Take for $\pi'$ an unramified principal series, for
example the trivial representation. As recalled above, compatibility with parabolic
induction and local class field theory implies that $\varphi_{\pi'}$ is unramified, hence
it is a geometric point of the tame summand $Z^{1}(W_{F}^{0},\HG)_{\rm tame}$, and it follows that the
same is true for $\varphi_{\pi}$.  The general case is similar but requires additional
notation, see \S \ref{pf_main}.

\medskip

Actually, when $\Gf$ is tamely ramified, further considerations based on congruences provide some extra
information about the Fargues-Scholze parameter $\varphi_{\pi}$ of any depth $0$ irreducible
representation $\pi$ over an algebraically closed field $L$ over
$\o\ZM[\frac 1p]$.
Namely, after choosing a generator $\tau$ of the tame
inertia group $I_{F}/P_{F}$, we get a
$\tau$-twisted conjugacy class $\varphi_{\pi}(\tau)$ in $\HG(L)$, which corresponds
to an $L$-point of the finite scheme $(\HG\rtimes\tau\sslash\HG)^{\Fr=(.)^{q}}$. On
the other hand, using type theory and Deligne-Lusztig series,  we can
directly associate to $\pi$ an $L$-point $s_{\pi}$ of the same finite
scheme 
(see \ref{sec:deligne-luszt-param}, note that this Deligne-Lusztig point also depends on $\tau$, now seen as
a choice of trivialization of roots of unity in $\o k_{F}^{\times}$).
It is then  expected that $\varphi_{\pi}(\tau) = s_{\pi}$. Our
considerations here provide the following modest contribution to this
question:

\begin{The}\label{Thm_order_param}
  Assume that $\Gf$ is quasi-split and tamely ramified, and that
  $k_{F}\neq \FM_{2}$. Then, with the above notation, any prime divisor $\ell$ of the order of
  $\varphi_{\pi}(\tau)$ divides the order of $s_{\pi}$. In particular,
  if $\pi$ is unipotent (i.e. $s_{\pi}=1$), then $\varphi_{\pi}$ is unramified.
\end{The}

Here is a sketch of the argument. Using parabolic induction and its compatibility with
the  map $\rm FS_{\rm mot}$, we may argue by induction on the rank of $\Gf$ and
focus on the case where $\pi$ is supercuspidal. A bit of
representation theory then allows us to reduce to the case $L=\o\QM$ and
$\pi$ having an admissible $\o\ZM[\frac 1p]$-model. Then we use
induction on the number of primes dividing the order
of $s_{\pi}$, and the crucial case is when this number is $1$, i.e.
when $\pi$ is unipotent.
In this case, suppose we can find two maximal ideals $\LC_{1}$ and
$\LC_{2}$ in $\o\ZM[\frac 1p]$ with respective residue characteristic
$\ell_{1}\neq \ell_{2}$, such that the block that contains
$\pi\, [\rm mod\,\LC_{i}]$ also contains a non-cuspidal unipotent
representation. Then, using our induction hypothesis and the
description of connected components of
$Z^{1}(W_{F}^{0},\HG)_{\o\ZM_{\ell}}$ in \cite[Thm 4.8]{DHKM}, we see
that both restrictions of $\varphi_{\pi}$ to $I_{F}^{\ell_{i}}$,
$i=1,2$ are trivial, and it follows that $\varphi_{\pi}$ is unramified.
Now, the problem of finding two such primes can be reduced to the
analogous problem for unipotent cuspidal representations of a finite
group of Lie type, where we check existence whenever the base field is
not $\FM_{2}$. The detailed argument is given in Section \ref{sec:proof-theor-refthm}.

\medskip

The two above results are consequences of a few formal properties of
the Fargues-Scholze construction, and say nothing on the non-triviality of the map
$\pi\mapsto \varphi_{\pi}$. In order to say anything on this matter, one obviously needs
to really work (hard!)  on the construction itself. Recently, Tony Feng proved a remarkable
property along these lines, that we extract from Theorem 10.4.1 of \cite{Feng} and the
subsequent remarks therein.
Assume that $\Gf(F)$ contains an  element of prime order $\ell$ whose
centralizer  is an unramified maximal  torus $\Tf$ of $\Gf$.
Let $L$ have characteristic $\ell$ and $s$ be an $L$-point of $(\HG\sslash\HG)^{\Fr=(.)^{q}}$
in the image of the map $\HT^{\Fr=(.)^{q}}\To{}(\HG\sslash\HG)^{\Fr=(.)^{q}}$.
Then there is $\pi\in\Irr_{L}(G)$ such
that $\varphi_{\pi}(\tau)=s=s_{\pi}$.  Moreover, if $s$ is strongly regular, then any
$\pi$ with $s_{\pi}=s$ satisfies $\varphi_{\pi}(\tau)=s_{\pi}$.  The techniques used by
Feng are bound to positive characteristic coefficients, but the last theorem above allows
to lift to characteristic $0$ under favorable circumstances.

\begin{Cor} \label{cor_Feng}
  With Feng's hypothesis above on $\Gf$ and $\Tf$, let $s$ be a $\o\QM$-point of $(\HG\sslash\HG)^{\Fr=(.)^{q}}$
  in the image of the map
  $\HT^{\Fr=(.)^{q}}\To{}(\HG\sslash\HG)^{\Fr=(.)^{q}}$
  and of order prime to $\ell$.
  Then there is $\pi\in\Irr_{\o\QM}(G)$ such
  that $\varphi_{\pi}(\tau)=s=s_{\pi}$.  Moreover, if $s$ is  strongly regular, then any
  $\pi$ with $s_{\pi}=s$ satisfies $\varphi_{\pi}(\tau)=s_{\pi}$.
\end{Cor}




\section{Finite groups}
\label{secFiniteGroups}

Let $\Gf$ be a reductive group over
$\o\FM_{p}$ and $\fr$ the Frobenius map associated to a $\FM_{q}$-rational structure on
$\Gf$, where $q=p^{r}$ for some $r\geq 1$.
The goal of this section is to prove the following theorem.
\begin{The} \label{primitive}
  The category $\Rep_{\o\ZM[\frac 1p]}(\Gf^{\fr})$ is indecomposable. Equivalently, the central idempotent $1$ in
  $\o\ZM[\frac 1p]\Gf^{\fr}$ is primitive.
\end{The}

\subsection{General facts  on  blocks of finite groups}
Let us start with an abstract finite group $G$. For any commutative ring $\Lambda$, denote
by $\Rep_{\Lambda}G$ the abelian category of $\Lambda G$-modules. Recall
that a block of $\Lambda G$ is an indecomposable direct summand of the category $\Rep_{\Lambda}G$. These
blocks correspond bijectively to indecomposable two-sided ideals of the ring $\Lambda G$ which are direct factors
and to primitive ideals in the center $Z(\Lambda G)$ of the ring $\Lambda G$, which is
also the center of the category $\Rep_{\Lambda}G$.
Note that the center $Z(\Lambda G)$ is the submodule of $G$-invariant elements in $\Lambda G$, so
that any base change map  $\Lambda'\otimes_{\Lambda}Z(\Lambda G)\To{} Z(\Lambda' G)$ is an isomorphism.

For a prime $\ell$, the center $Z(\o\ZM_{\ell} G)$ is finite over the Henselian local
ring $\o\ZM_{\ell}$, hence the reduction map
$Z(\o\ZM_{\ell} G)\To{}\o\FM_{\ell}\otimes_{\o\ZM_{\ell}}Z(\o\ZM_{\ell} G)=Z(\o\FM_{\ell}
G)$ induces a bijection on primitive idempotents, 
whence a
bijection between  blocks of $\o\FM_{\ell}G$ and of
$\o\ZM_{\ell}G$. The decomposition of $\Rep_{\o\ZM_{\ell}}G$ as a direct sum of blocks
induces a partition of the set $\Irr_{\o\Ql}(G)$ of isomorphism classes of simple
$\o\Ql G$-modules. The action of  the group of automorphisms of the field $\o\Ql$
on the group algebra $\o\Ql G$ preserves the integral algebra $\o\Zl G$, hence its action
on the set $\Irr_{\o\Ql}(G)$ preserves the above partition (permuting the parts).
It follows that this partition  can be transported
unambiguously to a partition of the set $\Irr(G)$ of irreducible complex representations
of $G$. Each factor set occurring in this partition will be called an $\ell$-block of
$\Irr(G)$.

Here is another point of view on $\ell$-blocks of $\Irr(G)$. In $\CM G$ we have the
decomposition $1=\sum_{\pi\in\Irr(G)}e_{\pi}$ of $1$, where
$e_{\pi}=\frac{\dim\pi}{|G|}\sum_{g\in G}\tr(\pi(g))g$ are the primitive central
idempotents of $\CM G$. As the formula shows, each $e_{\pi}$ belongs to $\o\ZM[\frac
    1{|G|}]G$ so that this decomposition of $1$ actually holds in $\o\ZM[\frac
    1{|G|}]G$. Denote by $|G|_{\ell'}$ the prime-to-$\ell$ factor of $|G|$, and declare that a
subset $I\subset \Irr(G)$ is $\ell$-integral if $\sum_{\pi\in I}e_{\pi}\in \o\ZM[\frac
  1{|G|_{\ell'}}]G$. Clearly,  $\ell$-integral subsets of $\Irr(G)$ are stable under taking unions,
intersections and complementary subsets.

\begin{Lem}
  The $\ell$-blocks of $\Irr(G)$ are the minimal non-empty $\ell$-integral subsets.
\end{Lem}
\begin{proof}
  Since both the property of being an $\ell$-block and of being $\ell$-integral are
  invariant under field automorphisms, we may transport the statement to $\Irr_{\o\Ql}(G)$
  where it follows from the fact that for any block $B$ of $\Rep_{\o\ZM_{\ell}}G$, the
  corresponding primitive central idempotent $e_{B}$ in $Z(\o\ZM_{\ell}G)$ is given by
  $e_{B}=\sum_{\pi\in \Irr_{\o\QM_{\ell}}(G)\cap B}e_{\pi}$. 
\end{proof}

Let us denote by $\sim_{\ell}$ the equivalence relation on $\Irr(G)$ whose equivalence
classes are the $\ell$-blocks. Now, fix a prime $p$, and denote by $\sim$ the equivalence
relation generated by all $\sim_{\ell}$ for $\ell\neq p$. Explicitly, we thus have
$\pi \sim \pi'$ if and only if there exist $\ell_1$,...,$\ell_r$ a sequence of primes
different from $p$ and $\pi_1,\cdots,\pi_{r-1} \in \Irr(G)$ such that
$$\pi \sim_{\ell_1} \pi_1 \sim_{\ell_2} \pi_2 \sim_{\ell_3} \cdots \sim_{\ell_r} \pi'.$$

\begin{Pro}	\label{proprimitif}
  The $\sim$-equivalence classes are the minimal non-empty subsets $I\subset \Irr(G)$ such
  that $e_{I}:=\sum_{\pi\in I}e_{\pi}\in \o\ZM[\frac 1p]G$. Moreover, the map $I\mapsto
    e_{I}$ is a bijection from $\Irr(G)/\sim$ onto the set of primitive idempotents of $\o\ZM[\frac 1p]G$.
\end{Pro}
\begin{proof}
  Follows from the previous lemma and the equality $\o\ZM[\frac 1p]\cap\o\ZM[\frac 1{|G|}]=\bigcap_{\ell\neq p}
    \o\ZM[\frac 1{|G|_{\ell'}}]$.
\end{proof}

Specializing our discussion to the case $G=\Gf^{\fr}$, the last proposition shows that
proving Theorem \ref{primitive} is equivalent to proving that there is only one
$\sim$-equivalence class in $\Irr(\Gf^{\fr})$.  Before doing so,
we need a recollection of Deligne-Lusztig theory.

\subsection{Blocks of finite reductive groups}
\label{secBlocksFiniteGroups}

Fix a  reductive group $(\Gf^{*},\fr^{*})$ over $\FM_{q}$ that is dual to $(\Gf,\fr)$ in
the sense of Lusztig.  Using their ``twisted'' induction functors, Deligne and Lusztig
define a partition $\Irr(\Gf^{\fr})=\bigsqcup_{s}
  \EC(\Gf^{\fr},s)$ of irreducible representations into ``Deligne-Lusztig series''
associated to semisimple elements $s$ of $\Gfd^{\frd}$ up to $\Gfd^{\frd}$-conjugacy.
Note that this partition depends on certain compatible choices of roots of unity, but
these choices will be irrelevant to our matters.

For a prime $\ell\neq p$ and a semisimple element $s\in \Gfd^{\frd}$ of order prime to
$\ell$,  we put
$$\EC_{\ell}(\Gf^{\fr},s):= \bigcup_{t_{\ell'}=s}\EC(\Gf^{\fr},t),$$ where
$t_{\ell'}$ denotes the prime-to-$\ell$ part of $t$ (i.e., we write
$t=t_{\ell'}t_{\ell}$ in $\langle t\rangle$, with $t_{\ell}$ of order a power of $\ell$ and
$t_{\ell'}$ of order prime to $\ell$). The following fundamental results are due to Broué
and Michel, resp. Hiss, and are stated in \cite[Thm. 9.12]{cabanes_enguehard}.

\begin{enumerate}
  \item $\dll{\Gf^{\fr}}{s}$ is a union of $\ell$-blocks.
  \item For each $\ell$-block $B$ such that $\Irr(\Gf^{\fr},B) \subseteq \dll{\Gf^{\fr}}{s}$, one has $\Irr(\Gf^{\fr},B) \cap \dl{\Gf^{\fr}}{s} \neq 0$.
\end{enumerate}

From these results we easily deduce the following fact.
\begin{Pro} \label{prounipotent}
  Any representation $\pi$ in $\Irr(\Gf^{\fr})$ is $\sim$-equivalent to a representation in $\dl{\Gf^{\fr}}{1}$.
\end{Pro}
\begin{proof}
  Let $s$ be a semisimple element in $\Gfd^{\frd}$ such that $\pi\in \EC(\Gf^{\fr},s)$,
  and let $\ell$ be a prime that divides the order of $s$. Note that $\ell\neq p$. As
  above, denote by $s_{\ell'}$ the prime-to-$\ell$ part of $s$. Then (2)
  above tells us that the $\ell$-block containing $\pi$ intersects
  $\EC(\Gf^{\fr},s_{\ell'})$. Therefore, the $\sim$-equivalence class of $\pi$
  intersects $\EC(\Gf^{\fr},s_{\ell'})$ too. But the order of $s_{\ell'}$ is the
  prime-to-$\ell$
  factor of the order of $s$. So, arguing inductively on the number of prime
  divisors of the order of $s$, we conclude that the $\sim$-equivalence class of $\pi$
  intersects $\EC(\Gf^{\fr},1)$.
\end{proof}

Let us now denote by $\simu$ the equivalence relation on $\dl{\Gf^{\fr}}{1}$ defined in
the same way as $\sim$, with every intermediate representation $\pi_i$  taken in
$\dl{\Gf^{\fr}}{1}$. By the last proposition, in order to prove Theorem \ref{primitive}, it
suffices to show that $\dl{\Gf^{\fr}}{1}$ has a unique $\simu$-equivalence class.

\begin{Pro} \label{proreduceadjoint}
  Let $\Gf_{\rm ad}$ be the adjoint group of $\Gf$, denote by $\pi :
    \Gf^{\fr}\To{}\Gf_{\rm ad}^{\fr}$ the natural map and by $\pi^{*}$ the associated
  pullback on representations. Then $\pi^{*}$ induces a
  bijection on unipotent representations $\dl{\Gf_{\rm ad}^{\fr}}{1} \simto
    \dl{\Gf^{\fr}}{1}$ that is compatible with $\simu$-equivalence on both sides.
\end{Pro}
\begin{proof}
  The fact that $\pi^{*}$ induces a
  bijection on unipotent representations
  is clear from the very definition of these representations.
  Moreover, for a prime $\ell$ different from $p$, \cite[Thm. 17.1]{cabanes_enguehard} tells us that
  $\Gf^{\fr}$ and ${\Gf_{\rm ad}}^{\fr}$ have
  the same number of unipotent $\ell$-blocks and that $\pi^{*} : \mathbb{Z}\dll{{\Gf_{\rm
            ad}}^{\fr}}{1}\To{} \mathbb{Z}\dll{\Gf^{\fr}}{1}$ preserves
  the orthogonal decomposition induced by $\ell$-blocks. It follows that the bijection
  $\pi^{*}:\dl{\Gf_{\rm ad}^{\fr}}{1} \simto
    \dl{\Gf^{\fr}}{1}$ on unipotent representations is compatible with the respective
  $\ell$-block partitions. Since the  $\sim^{1}$-equivalence classes are the subsets which
  are stable under $\sim_{\ell}$-equivalence for $\ell\neq p$ and minimal for this
  property,
  $\pi^{*}$ is also compatible with the partition into
  $\sim^{1}$-equivalence classes.
\end{proof}

This proposition allows us to reduce the general case to
the case where $\Gf$ is of adjoint type. But a group of adjoint type is a direct product
of restriction of scalars of simple groups \cite[Prop. 6.4.4 and Rk. 6.4.5]{conrad_luminy}.
So, in the sequel we may restrict attention to simple groups. It turns out that in some
cases, there is a quick argument using $2$-block theory.

\begin{The} \label{theqodd}
  Suppose $q$ is odd and  $\Gf$ has type $\Aa_n$,  ${}^{2}\Aa_n$,  $\Bb_n$, $\Cc_n$, $\Dd_n$
  or ${}^{2}\Dd_n$. Then $\dl{\Gf^{\fr}}{1}$ is composed of only one $\simu$-equivalence class.
\end{The}
\begin{proof}
  Indeed, by \cite[Thm. 21.14]{cabanes_enguehard}, $\dl{\Gf^{\fr}}{1}$ is included in the
  principal 2-block. So all the unipotent representations are already $\sim_2$-equivalent.
\end{proof}

In order to deal with $q$ even and the exceptional groups, we need to recall more results about blocks
that contain a unipotent representation.

\subsection{\texorpdfstring{$d$}{d}-series}
The unipotent $\ell$-blocks can be obtained using $d$-Harish-Chandra
theory, which provides a partition of $\EC(\Gf^{\fr},1)$ into \emph{$d$-series}, and where
$d\geq 1$ is an integer.  When $d=1$, the $1$-series are the usual Harish-Chandra
series constructed via parabolic induction.
In general, they are defined through an analogous pattern, relying on Deligne-Lusztig
induction and the following definitions.
\begin{itemize}
  \item An $\fr$-stable Levi subgroup of $\Gf$ is called a \emph{$d$-split Levi subgroup} if it is the
        centralizer of a $\Phi_{d}$-torus, i.e. an $\fr$-stable torus $\Sf$ such that
        $|\Sf^{\fr^{n}}|=\Phi_{d}(q^{n})$ for a certain $a \geq 1$ and all $n>0$ with
        $n\equiv 1 \pmod{a}$. Here, $\Phi_{d}$ denotes the $d$-cyclotomic polynomial.
  \item An irreducible representation
        $\pi\in\Irr(\Gf^{\fr})$ is called \emph{$d$-cuspidal} if for
        all proper $d$-split Levi subgroups $\Lf$
        and all parabolic subgroups $\Pf$ with Levi $\Lf$, the
        Deligne-Lusztig twisted restriction $^{*}R_{\Lf\subset \Pf}^{\Gf}[\pi]$ vanishes.
  \item A \emph{$d$-cuspidal pair} for $\Gf$ is a pair $(\Lf,\sigma)$ with $\Lf$ a $d$-split Levi subgroup
        of $\Gf$ and  $\sigma \in \Irr(\Lf^{\fr})$ $d$-cuspidal.
  \item The \emph{$d$-cuspidal support} of $\pi\in\Irr(\Gf^{\fr})$  is the set of all $d$-cuspidal
        pairs $(\Lf,\sigma)$ such that $\pi$ appears with non-zero multiplicity in
        the virtual character $R_{\Lf\subset \Pf}^{\Gf}[\sigma]$.
\end{itemize}

According to \cite[Thm 3.2]{bmm}, the $d$-cuspidal support of  any unipotent
$\pi\in\EC(\Gf^{\fr},1)$ consists of a single $\Gf^{\fr}$-conjugacy class of $d$-cuspidal
pairs $(\Lf,\sigma)$.  We thus get a partition of $\EC(\Gf^{\fr},1)$ labeled by conjugacy
classes of unipotent $d$-cuspidal pairs. The summands appearing in this partition are called \emph{$d$-series}.
The remarkable relevance of $d$-series to the study of $\ell$-blocks is summarized in the
following statement.

\begin{The}[{\cite[Thm. 4.4]{CabanesEngueharduni}}]
  \label{thelblock}
  Let $\ell$ be a prime not dividing $q$ and let $d$ be the order of $q$ in $\FM_{\ell}^{\times}$.
  We assume that $\ell$ is odd, good for $\Gf$, and $\ell
    \neq 3$ if ${}^3\Dd_4$ is involved in $(\Gf,\fr)$.
  Then, the map $B\mapsto B\cap \EC(\Gf^{\fr},1)$ induces a bijection
  \begin{center}
    $\{\ell$-blocks $B\subset \EC_{\ell}(\Gf^{\fr},1)\}$ $\simto$ $\{d$-series in $\EC(\Gf^{\fr},1)\}$.
  \end{center}

\end{The}

This theorem suggests the following strategy to prove that $\EC(\Gf^{\fr},1)$ has only one
$\sim^{1}$-equivalence class.  If $D\subset \NM^{*}$ is a finite set of non-zero integers,
define a \emph{$D$-series} as a subset of $\EC(\Gf^{\fr},1)$ that is a union of $d$-series for
each $d\in D$, and which is \emph{minimal} for this property.
\begin{Lem} \label{lemDseries} Suppose there exists $D$ such that
  \begin{enumerate}
    \item $\EC(\Gf^{\fr},1)$ is a $D$-series.
    \item Each $d\in D$ is the order of $q$ modulo some $\ell$ as in  Theorem
          \ref{thelblock}.
  \end{enumerate}
  Then $\EC(\Gf^{\fr},1)$ consists of  a unique $\simu$-equivalence class.
\end{Lem}
\begin{proof}
  Use (2) to pick a prime $\ell_{d}$ satisfying the assumptions of Theorem
  \ref{thelblock} and such that $q$ has order $d$ modulo $\ell_{d}$, for each $d\in
    D$. Then (1) tells us that  $\EC(\Gf^{\fr},1)$ is the only non-empty subset of itself
  that is stable under $\sim_{\ell_{d}}$-equivalence for all $d\in D$. A fortiori, it is
  the only non-empty subset of itself
  that is stable under $\sim_{\ell}$-equivalence for all primes $\ell\neq p$. Hence it is
  a single $\simu$-equivalence class.
\end{proof}

Finding a suitable $D$ will be done below via a case-by-case analysis (recall that we have
reduced to the case where $\Gf$ is simple). Then, in order to find suitable primes, the
following result will be useful.

\begin{The}[{\cite[Thm. V]{BirkVand}}]
  \label{theordre}
  For any $d\geq 3$, there exists a prime number $\ell$ such that $q$ has order $d$ modulo
  $\ell$, with the exception of $(q,d)=(2,6)$.
\end{The}
Note that if $d$ is the order of $q$ modulo $\ell$, then
obviously $\ell > d$.
We now proceed to the case-by-case analysis.

\begin{The}
  \label{theAn}
  If $\Gf$ has type $\Aa_n$ or ${}^{2}\Aa_n$ then $\dl{\Gf^{\fr}}{1}$ is composed of only one $\simu$-equivalence class.
\end{The}

\begin{proof}

  The case $q$ odd is covered by Theorem \ref{theqodd}, so let us assume $q$ even.
  Since $q+1$ is odd,  any prime divisor $\ell$ of $q+1$ is odd and good for $\Gf$, and the
  order of $q$ modulo $\ell$ is $2$. Therefore Theorem
  \ref{thelblock} tells us that the $\ell$-blocks of $\EC(\Gf^{\fr},1)$ are the
  $2$-series.

  Now, for type $\Aa_{n}$ (split case), it is well known that $\EC(\Gf^{\fr},1)$ is a single $1$-series. Now for type $^{2}\Aa_{n}$ we can use an  ``Ennola duality'' \cite[Thm 3.3]{bmm}  to reduce ourselves to the case of $\Aa_{n}$. More specifically, \cite[Thm 3.3]{bmm} implies that there is a bijection between the set of  unipotent characters for a group of type $\Aa_{n}$ and that of $^{2}\Aa_{n}$. This bijection sends $d$-series to $d'$-series, where $d'$ is such that $\Phi_{d'}(-x) =\varepsilon \Phi_{d}(x)$ with $\varepsilon \in \{\pm 1\}$. Since $\Phi_{2}(-x) = -x + 1 = - \Phi_{1}(x)$, it sends $1$-series to $2$-series. It follows that in type $^{2}\Aa_{n}$, the set $\EC(\Gf^{\fr},1)$ is a single $2$-series. Therefore it is a single $\ell$-block
  for $\ell|(q+1)$ hence also a single  $\sim^{1}$-equivalence
  class.

  Similarly, when $q>2$,  any prime $\ell$ dividing $q-1$ is odd and good for $\Gf$, hence Theorem
  \ref{thelblock} tells us that the $\ell$-blocks of $\EC(\Gf^{\fr},1)$ for $\ell|(q-1)$ are the
  $1$-series. In type $\Aa_{n}$ and for such an $\ell$, it follows that the set $\EC(\Gf^{\fr},1)$ is a single
  $\ell$-block hence also a single  $\sim^{1}$-equivalence class.

  It remains to deal with the case $q=2$ in type $\Aa_{n}$.
  In this case, we will show that $\dl{\Gf^{\fr}}{1}$ is a
  $\{2,3\}$-series, and then we can conclude  using Theorem \ref{theordre} and Lemma \ref{lemDseries}.
  To compute $\{2,3\}$-series for $\Aa_{n}$, we again use Ennola duality, which asserts that they
  correspond bijectively to $\{1,6\}$-series for $^{2}\Aa_{n}$. These $\{1,6\}$-series have been
  computed in \cite[Section 3.3]{lanardlblocs}, to
  which we refer for the notion of ``defect'' of a $1$-series.
  In particular, \cite[Prop. 3.3.11]{lanardlblocs} shows that there is a unique
  $\{1,6\}$-series for $^{2}\Aa_{n}$ provided we can prove that the defect $k$ of any
  1-series of $^{2}\Aa_{n}$   satisfies $(k^2-3k+2)/2\leq n-2$. But by
  \cite[Lem. 3.3.8]{lanardlblocs}, we at least have
  $k(k+1)/2 \leq n+1$. Since $(k^2-3k+2)/2 = k(k+1)/2 -(2k-1)$, we thus get
  the desired equality if $2k-1 \geq 3$, that is $k\geq 2$. Moreover, for $k=1$, we have
  $(k^2-3k+2)/2=0$, thus the desired inequality also holds.
\end{proof}

\begin{The}
  \label{theclassic}
  If $\Gf$ has type $\Bb_n$, $\Cc_n$, $\Dd_n$ or ${}^{2}\Dd_n$ then $\dl{\Gf^{\fr}}{1}$ is composed of only one $\simu$-equivalence class.
\end{The}

\begin{proof}
  Again, the case $q$ odd is covered by Theorem \ref{theqodd}, so we assume $q$ is even.
  Picking a divisor of $q+1$ and applying Theorem \ref{theordre} to $d=4$, we see thanks to
  Lemma \ref{lemDseries} that it suffices to prove that $\EC(\Gf^{\fr},1)$ is a single $\{2,4\}$-series.

  We will first exhibit a bijection between $\{2,4\}$-series and $\{1,4\}$-series, which will
  leave us with actually proving that $\EC(\Gf^{\fr},1)$ is a single $\{1,4\}$-series

  To do so, we use the combinatorics of Lusztig symbols as in \cite[Section 3.4]{lanardlblocs}.
  Let $\Sigma=\{S,T\}$ be a symbol ($S,T \subseteq \mathbb{N}$). We define $S_e$ to be the subset of $S$ composed of the even elements, that is $S_e:=S \cap 2\mathbb{N}$, and $S_o$ for the subset of odd elements, $S_o:=S \cap (2\mathbb{N}+1)$. We do the same thing for $T$, $T=T_e \cup T_o$. Now, we define an involution $\varphi$ on symbols by $\varphi(\Sigma):=\{S_e \cup T_o,T_e \cup S_o\}$. Note that $\rank(\Sigma)=\rank(\varphi(\Sigma))$ and that $\defect(\Sigma)$ and $\defect(\varphi(\Sigma))$ have the same parity. In the case of $\Dd_n$ or ${}^{2}\Dd_n$, when the defect is even, the congruence modulo 4 is not necessarily preserved by $\varphi$. However, the congruence mod 4 of $\defect(\varphi(\Sigma))$ only depends on the congruence modulo 4 of $\defect(\Sigma)$ and $\rank(\Sigma)$. Indeed
  $$\defect(\varphi(\Sigma)) - \defect(\Sigma) =| |S_e|+|T_o|-|S_o|-|T_e|| - | |S_e|+|S_o|-|T_e|-|T_o||.$$
  Thus, the congruence modulo 4 of $\defect(\varphi(\Sigma)) - \defect(\Sigma)$ depends on the parity of $|S_o|+|T_o|$ (or $|S_e|+|T_e|$). Since the defect is even, $|S|+|T|$ is even, hence $|S_o|+|T_o| \equiv |S_e|+|T_e| \pmod{2}$. Now,
  $$\rank(\Sigma)=\sum_{x\in S} x + \sum_{y\in T} y + \left[\left(\frac{|S|+|T|-1}{2}\right)^2\right],$$
  thus
  $$\rank(\Sigma)\equiv  |S_o|+|T_o| + \left[\left(\frac{|S|+|T|-1}{2}\right)^2\right]  \pmod{2}.$$
  But, $|S|+|T|$ is even, $|S|+|T|=2k$ and $[((|S|+|T|-1)/2)^2]=k(k-1)$ is even. Thus $\defect(\varphi(\Sigma)) - \defect(\Sigma)$ only depends modulo 4 on the parity of $\rank(\Sigma)$.

  Now, under the involution $\varphi$, 1-hooks correspond to 1-co-hooks and
  2-co-hooks to 2-co-hooks. Hence, $\varphi$ sends $\{2,4\}$-series to $\{1,4\}$-series,
  as claimed above.

  We now prove that $\dl{\Gf^{\fr}}{1}$ is a $\{1,4\}$-series in the different cases.

  For $\Gf$ of type $\Bb_n$ or $\Cc_n$, by \cite[Prop. 3.4.6]{lanardlblocs}, we need
  to show that if we have a 1-series of defect $k$ then $(k^2-4k+3)/4\leq n-2$. If
  we have a 1-series of defect $k$ and rank $n$ then $(k^2-1)/4 \leq n$. Since
  $(k^2-4k+3)/4 = (k^2-1)/4 -(k-1)$, we get the desired equality if $k-1 \geq 2$,
  that is $k\geq 3$. For $k=1$, $(k^2-4k+3)/4=0$, thus the desired inequality also holds.

  For $\Gf$ of type $\Dd_n$ or ${}^{2}\Dd_n$, by \cite[Prop. 3.4.6]{lanardlblocs},
  we need to show that if we have a 1-series of defect $k$ then $(k^2-4k+4)/4\leq
    n-2$. If we have a 1-series of defect $k$ and rank $n$ then $k^2/4 \leq n$. Since
  $(k^2-4k+4)/4 = k^2/4 -(k-1)$, we get the desired equality if $k-1 \geq 2$, that
  is $k\geq 3$. For $k=2$, $(k^2-4k+3)/4=0$, thus the desired inequality still holds.
\end{proof}

\begin{The}
  \label{theExeptionel}
  If $\Gf$ has type $\Ff_4$, ${}^{3}\Dd_4$, $\Gg_2$, $\Ee_6$, ${}^{2}\Ee_6$, $\Ee_7$ or $\Ee_8$ then $\dl{\Gf^{\fr}}{1}$ is composed of only one $\simu$-equivalence class.
\end{The}

\begin{proof}
  First, let us begin by $\Gf$ of type ${}^{3}\Dd_4$, $\Ee_6$, ${}^{2}\Ee_6$ or $\Ee_7$. To
  use Theorem \ref{thelblock}, we need to have $\ell \geq 5$. Thus by Theorem \ref{theordre}
  we can use every $d\geq 3$ and $d\neq 6$.
  Looking at the list of unipotent characters in \cite[\S 13.9]{carter} and tables of
  $d$-series in \cite{bmm}, we see that we may apply Lemma \ref{lemDseries} with the
  following sets:  $D=\{3,12\}$ for ${}^{3}\Dd_4$, $D=\{3,4\}$ for $\Ee_{6}$, $D=\{3,4,12,18\}$ for $^{2}\Ee_{6}$ and $D=\{3,4,14\}$ for $\Ee_7$.

  \medskip

  For $\Ee_8$, we need $\ell \geq 7$ to use Theorem \ref{thelblock}. We will again conclude with Lemma \ref{lemDseries} by looking at tables. If $q \neq 2$, we can take $d\geq 5$ by Theorem \ref{theordre} and $D=\{5,6,7,8,10,30\}$ works. If $q=2$, we can take this time $d\geq 3$ $d\neq 4,6$ (since the order of 2 modulo 7 is 3) and we choose $D=\{3,5,8,10,15\}$.

  \medskip

  For $\Ff_4$, the same methods apply for $q\neq 2$ with $D=\{3,4,6,12\}$ (we can choose any $d\geq 3$). When $q=2$, Lemma \ref{lemDseries} is not enough to conclude. The set of $d$ such that $q$ is of order $d$ modulo some $\ell \geq 3$ is $d \geq 3$ and $d \neq 6$ by Theorem \ref{theordre}. However, the two unipotent characters $\phi_{9,6}'$ and $\phi_{9,6}''$ (with the notations of \cite[Section 13.9]{carter}) are $d$-cuspidal for every $d \geq 3$, $d \neq 6$. The rest of the unipotent characters $\EC(\Gf^{\fr},1) \setminus \{\phi_{9,6}',\phi_{9,6}''\}$ form a $\{3,4,8,12\}$-series. To deal with $\phi_{9,6}'$ and $\phi_{9,6}''$, we take $\ell=3$, and they are in the principal $3$-block of $\Ff_4(2)$ by \cite{hissF4}.

  \medskip

  We are left with $\Gg_2$. If $q$ is odd, we can take $d\geq 3$. The two unipotent characters  $\phi_{1,3}'$ and $\phi_{1,3}''$ are $d$-cuspidal for every $d \geq 3$ and $\EC(\Gf^{\fr},1) \setminus \{\phi_{1,3}',\phi_{1,3}''\}$ is a $\{3,6\}$-series. We take $\ell=2$ and the tables in \cite{HissShamash2} give us that $\phi_{1,3}'$ and $\phi_{1,3}''$ are in the principal $2$-block. This concludes the case $q$ odd. Now, if $q$ is even and $q\neq 2$, we can still use $d\geq 3$. Thus, we have the same issue with $\phi_{1,3}'$ and $\phi_{1,3}''$. We can no longer use $\ell=2$ but we can use $\ell =3$, and the tables in \cite{HissShamash3} gives us that $\phi_{1,3}'$ and $\phi_{1,3}''$ are in the principal $3$-block. Finally, if $q=2$, we can now use $d\geq 3$ and $d\neq 6$. All the unipotent characters are in the principal $3$-series apart from $\phi_{1,3}'$, $\phi_{1,3}''$, $\phi_{2,1}$ and $\Gg_2[-1]$. But we see in \cite{HissShamash3} that they all are in the principal $3$-block.

\end{proof}

We have now completed the proof of Theorem \ref{primitive} for all
reductive $(\Gf,\fr)$. Indeed, by Proposition \ref{proprimitif},
the statement of this theorem is equivalent to the
statement that there is only one  $\sim$-equivalence class in $\Irr(\Gf^{\fr})$. Then
Proposition \ref{prounipotent} shows it is enough to prove that there is only one
$\sim^{1}$-equivalence class in $\EC(\Gf^{\fr},1)$, and Proposition
\ref{proreduceadjoint} reduces to the case of simple (and adjoint) $\Gf$. All the simple
cases are then covered by Theorems \ref{theqodd}, \ref{theAn}, \ref{theclassic} and \ref{theExeptionel}.

\section{\texorpdfstring{$p$}{p}-adic groups}

Here $\Gf$ is a reductive group over a local non-archimedean field $F$
with residue field $k_{F}:=\FM_{q}$. We put $G:=\Gf(F)$. For any
commutative ring $R$ in which $p$ is invertible, we denote by
$\Rep_{R}(G)$ the category of smooth $R G$-modules. The Bernstein center
$\ZG_{R}(G)$  is by definition the center of this category.

\subsection{The depth \texorpdfstring{$0$}{0} summand} \label{sec:depth-0-summand}
Denote by $\BT$ the (reduced) Bruhat-Tits building
associated with $\Gf$. This is a polysimplicial complex equipped with a polysimplicial
action of $G$. We will write $\BT_{\bullet}$ for the set of polysimplices of $\BT$, which
are also called
facets.
To any facet $\sigma$ of $\BT$ is associated a parahoric subgroup
$G_{\sigma}$, which is open, compact and contained in the pointwise stabilizer of
$\sigma$. It is the group $\Gf_{\sigma}(\OC_{F})$ of $\OC_{F}$-valued points of a certain
smooth $\OC_{F}$-model $\Gf_{\sigma}$ of $\Gf$. We denote by $\o\Gf_{\sigma}$ the reductive quotient of
the special fiber of $\Gf_{\sigma}$. Then $\o G_{\sigma}:=\o\Gf_{\sigma}(\FM_{q})$ is also
the quotient of $G_{\sigma}$ by its pro-$p$-radical $G_{\sigma}^{+}$. Since
$G_{\sigma}^{+}$ is open and pro-$p$, there is an averaging idempotent $e_{\sigma}^{+}\in
  \HC_{R}(G_{\sigma})\subset \HC_{R}(G)$ in the Hecke algebra $\HC_{R}(G)$ of $G$ with
coefficients in $R$. Here $\HC_{R}(G)$ denotes the $R$-algebra of locally constant
$R$-valued distributions on $G$, which acts on any smooth  $RG$-module $V$, and $e_{\sigma}^{+}$
denotes the distribution that averages a function over $G_{\sigma}^{+}$.
\begin{Def}
  A smooth $RG$-module $V$ has \emph{depth 0} if $V=\sum_{x\in \BT_{0}} e_{x}^{+}V$.
\end{Def}
Here $\BT_{0}$ is the set of vertices of $\BT$.
It is known \cite[Appendice A]{datFinitude} that the full
subcategory $\Rep^{0}_{R}(G)$ of $\Rep_{R}(G)$ composed of depth $0$ objects is a direct
factor abelian subcategory. Correspondingly, there is an idempotent
$\varepsilon_{0}\in\ZG_{R}(G)$ that projects any object $V$ onto its depth $0$ factor.
When $R=\CM$, the Bernstein decomposition of $\Rep^{0}_{\CM}(G)$ as a sum of blocks was
made explicit by Morris in \cite{morris}.

\subsection{(un)refined depth \texorpdfstring{$0$}{0} types}\label{sec:unref-depth-texorpdf}
Following Moy and Prasad, we define an \emph{unrefined depth
  $0$ type} to be a pair $(\sigma,\pi)$, where $\sigma \in \BT_{\bullet}$ is a facet and $\pi$ is an
irreducible complex cuspidal representation of $\o G_{\sigma}$. We also denote by $\pi$
the inflation of $\pi$ to $G_{\sigma}$. Then \cite[Theorem 4.8]{morris} tells us that
${\rm ind}_{G_{\sigma}}^{G}(\pi)$ is a projective generator for a certain sum of Bernstein
components of depth $0$. Obviously, this representation only depends on the $G$-conjugacy
class $\tG$ of $(\sigma,\pi)$,  whence a direct factor $\Rep_{\CM}^{\tG}(G)$ of
$\Rep_{\CM}^{0}(G)$.  Denote by $\TG$ the set of $G$-conjugacy classes of such pairs. It
turns out that, for $\tG,\tG'\in\TG$, the factors $\Rep_{\CM}^{\tG}(G)$ and
$\Rep_{\CM}^{\tG'}(G)$ are either orthogonal or equal. Whence an equivalence relation
$\sim$ on $\TG$ and a decomposition
$$ \Rep^{0}_{\CM}(G) = \prod_{[\tG]\in \TG/\sim} \Rep^{[\tG]}_{\CM}(G).$$
We refer to \cite[Section 2.2]{lanardlblocs} for an explicit description of the relation $\sim$.

In general, the factors $\Rep^{[\tG]}_{\CM}(G)$ are further decomposable. Suppose
$(\sigma,\pi)\in[\tG]$ and denote by $G_{\sigma}^{1}$ the  pointwise
stabilizer of $\sigma$ in the closed subgroup $G^{1}$ of $G$ generated
by compact subgroups (see \cite[\S 7.7]{kaletha_prasad}). By \cite[Theorem 4.9]{morris}, for any irreducible subquotient
$\pi^{1}$ of ${\rm ind}_{G_{\sigma}}^{G_{\sigma}^{1}}(\pi)$, the
representation ${\rm ind}_{G_{\sigma}^{1}}^{G}(\pi^{1})$
is the projective generator
of a single Bernstein block. Also, any block inside $\Rep^{[\tG]}_{\CM}(G)$ is
obtained in this way. For later reference, let us notice that, since
$G_{\sigma}^{1}/G_{\sigma}$ is abelian, any other irreducible subquotient of
${\rm ind}_{G_{\sigma}}^{G_{\sigma}^{1}}(\pi)$ is a twist of $\pi^{1}$ by a
character of $G_{\sigma}^{1}/G_{\sigma}$.

\subsection{Systems of idempotents}
Fix $\mathfrak{t} \in \TG$, and let $[\mathfrak{t}]$ denote its equivalence class.
For a facet $\tau\in\BT_{\bullet}$, define $e_{[\tG],\tau}\in\HC_{R}(G_{\tau})$ to be the idempotent
that cuts out all irreducible representations of $\o G_{\tau}$ whose cuspidal support
contains $(\o G_{\sigma},\pi)$ for some facet $\sigma$ containing $\tau$ and $\pi$ such
that $(\sigma,\pi)\in[\tG]$. Then the system $(e_{[\tG],\tau})_{\tau\in \BT_{\bullet}}$ is
\emph{$0$-consistent} in the sense of \cite[Def. 2.1.4]{lanardlblocs},
that is :
\begin{enumerate}
  \item $\forall x\in\BT_{0}$, $\forall g\in G$, $e_{[\tG], gx}=ge_{[\tG],x}g^{-1}$
  \item $\forall \tau\in\BT_{\bullet}$, $\forall x\in\BT_{0}$,
        $x\in \bar\tau\Rightarrow e_{[\tG],\tau}=e_{\tau}^{+}e_{[\tG],x}$.
\end{enumerate}
Moreover, Proposition
2.2.3 of \cite{lanardlblocs} implies that
$$ \Rep_{\CM}^{[\tG]}(G)= \left\{V\in\Rep_{\CM}(G),\, V=\sum_{x\in \BT_{0}} e_{[\tG],x}V\right\}.$$
Denote by $N$ the l.c.m of all $|\o G_{\tau}|$, $\tau\in\BT_{\bullet}$. Since each
$e_{[\tG],\tau}$ lies in $\HC_{\o\ZM[\frac 1N]}(G)$, we see that the summand
$\Rep^{[\tG]}_{\CM}(G)$ is ``defined over $\o\ZM[\frac 1N]$'' in the sense that the
corresponding idempotent $\varepsilon_{[\tG]}$ of $\ZG_{\CM}(G)$ lies in $\ZG_{\o\ZM[\frac
      1N]}(G)$, and we have a decomposition
\[
  \rep[\o\ZM[\frac 1N]][0]{G}=\prod_{[\mathfrak{t} ]\in \TG/{\sim}}\rep[\o\ZM[\frac 1N]][[\mathfrak{t}]]{G}.
\]

Now, to a subset $T$ of $\TG/\sim$  we associate an idempotent  $\varepsilon_{T}:=\sum_{[\tG]\in
    T}\varepsilon_{[\tG]} \in \ZG_{\o\ZM[\frac 1N]}(G)$ in the Bernstein center, and a
consistent system of idempotents $(e_{T,\tau})_{\tau\in\BT_{\bullet}}$ given by  $e_{T,\tau}:=\sum_{[\tG]\in
    T}e_{[\tG],\tau}$ for any facet $\tau\in \BT_{\bullet}$.
The following observation is crucial for the argument.
\begin{Pro}
  \label{proeTintegral}
  We have $\varepsilon_{T}\in\ZG_{\o\ZM[\frac 1p]}(G) \Leftrightarrow \forall\tau\in\BT_{\bullet}$, $e_{T,\tau}\in \HC_{\o\ZM[\frac 1p]}(G_{\tau})$.
\end{Pro}
\begin{proof}
  As recalled above, the direct factor category
  associated to the idempotent $\varepsilon_{T}$ is given by
  $$\varepsilon_{T}\Rep_{\o\ZM[\frac 1N]}(G) =  \left\{V\in\Rep_{\o\ZM[\frac 1N]}(G),\,
    V=\sum_{x\in \BT_{0}} e_{T,x}V\right\}.$$
  Moreover, its orthogonal complement in $\Rep^{0}_{\o\ZM[\frac 1N]}(G)$ is the category
  similarly associated to the complement subset $T^{c}$ of $T$ in $\TG/\sim$.
  Therefore, if all $e_{T,\tau}$ are
  $\o\ZM[\frac 1p]$-valued, then so are all $e_{T^{c},\tau}=e_{\tau}^{+}-e_{T,\tau}$, so the
  decomposition  $\Rep_{\o\ZM[\frac 1N]}^{0}(G)=\varepsilon_{T}\Rep_{\o\ZM[\frac
        1N]}(G)\times \varepsilon_{T^{c}}\Rep_{\o\ZM[\frac 1N]}(G)$ is defined over $\o\ZM[\frac
      1p]$, hence the idempotents $\varepsilon_{T}$ and $\varepsilon_{T^{c}}$ belong to
  $\ZG_{\o\ZM[\frac 1p]}(G)$.
  Conversely, if $\varepsilon_{T}\in\ZG_{\o\ZM[\frac 1p]}(G)$, then the
  equality  $e_{T,\tau}=\varepsilon_{T}* e_{\tau}^{+}$ in $\HC_{\CM}(G)$ from the proof of Lemma 4.1.2 in
  \cite{lanardlblocs}, shows that $e_{T,\tau}\in\HC_{\o\ZM[\frac 1p]}(G)$.
\end{proof}

Recall the notation $\varepsilon_{0}$ introduced in \ref{sec:depth-0-summand} for the central idempotent corresponding to the
depth $0$ factor.

\begin{Cor} \label{coridemp0}
  We have $\varepsilon_{T}\in\ZG_{\o\ZM[\frac 1p]}(G)$ if and only if $T=\TG/\sim$,
  i.e. if and only if $\varepsilon_{T}=\varepsilon_{0}$.
\end{Cor}
\begin{proof} Suppose $\varepsilon_{T}\in\ZG_{\o\ZM[\frac 1p]}(G)$. By the last proposition,
  for any facet $\tau$ in $\BT_{\bullet}$, the idempotent $e_{T,\tau}$ is inflated from a central
  idempotent of  $\o\ZM[\frac 1p]  \o G_{\tau}$. By Theorem \ref{primitive}, there is
  no non-trivial central idempotent in $\o\ZM[\frac 1p] \o G_{\tau}$. So $e_{T,\tau}$ is
  either  $e_{\tau}^{+}$ or $0$. There is  certainly a vertex $x$ such that $e_{T,x}\neq
    0$, and thus $e_{T,x}=e_{x}^{+}$. Then, for a chamber $\sigma$ containing $x$, we have
  $e_{T,\sigma}=e_{\sigma}^{+}e_{T,x}=e_{\sigma}^{+}$. By $G$-equivariance it follows that
  $e_{T,\tau}=e_{\tau}^{+}$ for all chambers, which implies $e_{T,\tau}\neq 0$ and
  therefore $e_{T,\tau}=e_{\tau}^{+}$ for all facets. Hence
  $\varepsilon_{T}=\varepsilon_{0}$ and $T=\TG/\sim$.
\end{proof}

When $G$ is semi-simple and simply-connected, we have $G_{\tau}=G_{\tau}^{1}$ for all facets $\tau$
so, from our discussion of depth $0$ types above, each $\Rep^{[\tG]}_{\CM}(G)$
is already a block of $\Rep^{0}_{\CM}(G)$ and, therefore, any idempotent of
$\varepsilon_{0}\ZG_{\CM}(G)$ is equal to some $\varepsilon_{T}$ for $T\subset
  \TG/\sim$. So we have proved :

\begin{Cor}
  If $G$ is semisimple and simply-connected, $\rep[\o\ZM[1/p]][0]{G}$ is a block
  (equivalently, $\varepsilon_{0}$ is primitive in $\ZG_{\o\ZM[\frac 1 p]}(G)$).
\end{Cor}

\begin{Rem} \label{remblocksequences}
  As in the case of finite groups, this kind of results can be interpreted at the level of
  irreducible $\o\QM G$-modules in the following way.
  For $M\in\NM^{*}$ and $\pi\in \Irr_{\o\QM}(G)$, denote by
  $\varepsilon_{M,\pi}$ the unique  primitive idempotent of
  $\varepsilon_{0}\ZG_{\o\ZM[\frac 1M]}(G)$  that acts as identity on $\pi$.  So, if $M=N$
  (defined as above), $\varepsilon_{N,\pi}$ is also primitive in $\ZG_{\o\QM}(G)$ and
  defines the Bernstein component that contains $\pi$. Let $\ell\neq p$ be a prime and
  denote by $N_{\ell'}$ the prime-to-$\ell$ part of $N$. Then, for
  $\pi,\pi'\in\Irr_{\o\QM}(G)$, the following properties are equivalent :
  \begin{enumerate}
    \item $\varepsilon_{N_{\ell'},\pi}=\varepsilon_{N_{\ell'},\pi'}$
    \item For any embedding $\o\QM\injo\o\QM_{\ell}$,  the base changed representations
          $\pi,\pi'$ belong to the same block of $\Rep_{\o\ZM_{\ell}}(G)$.
  \end{enumerate}
  This justifies calling the equivalence classes for the relation
  $\pi\sim_{\ell}\pi'\Leftrightarrow \varepsilon_{N_{\ell'},\pi}=\varepsilon_{N_{\ell'},\pi'}$ the \emph{$\ell$-blocks of
    $\Irr_{\o\QM}(G)$}. They correspond to minimal ``$\ell$-integral subsets'' of the set
  of Bernstein blocks. Similarly, the blocks of $\Rep_{\o\ZM[\frac 1p]}(G)$ correspond to
  minimal subsets of the set of Bernstein blocks that are $\ell$-integral for all $\ell\neq
    p$.
  In other words, the equivalence relation $\sim$ generated on $\Irr_{\o\QM}(G)$ by all
  $\sim_{\ell}$, $\ell\neq p$ satisfies $\pi\sim\pi'\Leftrightarrow$ $\pi$ and $\pi'$
  belong to the same block of $\Rep^{0}_{\o\ZM[\frac 1p]}(G)$.

\end{Rem}

Now, to tackle the general case, we need to recall some facts about the quotients $G_{\tau}^{1}/G_{\tau}$.

\subsection{The Kottwitz map}
We first recall a definition of Borovoi :
$$\pi_{1}(\Gf) := X_{*}(\Tf)/\langle\Phi^{\vee}\rangle = {\rm coker}(X_{*}(\Tf_{\rm sc})\To{}X_{*}(\Tf)).$$
Here $\Tf$ is a maximal torus of $\Gf$,  $\Phi^{\vee}\subset X_{*}(\Tf)$
is the set of (absolute) coroots, and $\Tf_{\rm sc}$ is the inverse image of $\Tf$ in the
simply connected covering $\Gf_{\rm sc}$ of the derived group $\Gf_{\rm der}$ of
$\Gf$. Using the fact that all tori are conjugate, the group $\pi_{1}(\Gf)$ turns out to be canonically
independent of the choice of $\Tf$. Moreover, if $\Tf$ is chosen so as to be defined over
$F$, then $\pi_{1}(\Gf)$ gets a $\ZM$-linear action of the Galois group $\Gamma_{F}={\rm Gal}(\o
  F/F)$. Again, this action does not depend on the choice of $F$-rational torus $\Tf$
(although two such choices may not be $\Gf(F)$-conjugate).

Now, let $I_{F}\subset \Gamma_{F}$ denote the inertia subgroup and let $\fr$ denote the
geometric Frobenius in $\Gamma_{F}/I_{F}$. Kottwitz has defined a surjective morphism
$$ \kappa_{G}:\,\, G=\Gf(F)\To{} \pi_{1}(G):=\left(\pi_{1}(\Gf)_{I_{F}}\right)^{\fr}.$$
We refer to \cite[Chap. 11]{kaletha_prasad} for the detailed construction of this map.
The following properties of this map are particularly relevant to our problem :
\begin{itemize}
  \item The kernel $G^{0}:=\ker \kappa_{G}$ is the subgroup of $G$ generated by parahoric
        subgroups and, for any facet $\sigma\in \BT_{\bullet}$, the parahoric group $G_{\sigma}$ is the
        pointwise stabilizer of $\sigma$ in $G^{0}$ (and actually also
        the stabilizer).
  \item The inverse image $G^{1}:=\kappa_{G}^{-1}(\pi_{1}(G)_{\rm tors})$ is the subgroup of
        $G$ generated by compact subgroups and, for any facet $\sigma\in \BT_{\bullet}$, the compact
        open group   $G_{\sigma}^{1}$ introduced above 
        is the pointwise stabilizer of $\sigma$ in $G^{1}$.
\end{itemize}
In particular, we have $G_{\sigma}^{1}\subset G^{1}$ and $G_{\sigma}=G^{0}\cap G_{\sigma}^{1}$.

Now let $\Psi_{G}$ be the diagonalizable algebraic group scheme over
$\o\ZM[\frac 1p]$ associated to
the finitely generated abelian group $\pi_{1}(G)$.
Its maximal subtorus $\Psi_{G}^{t}$ is the usual ``torus of unramified characters''
of $G$, while the quotient
$\Psi_{G}^{f}:=\Psi_{G}/\Psi_{G}^{t}$ is the diagonalizable group
scheme associated  to the finite group $\pi_{1}(G)_{\rm tors}$.

For any $\o\ZM[\frac 1p]$-algebra $R$, the group $\Psi_{G}(R)=\Hom(\pi_{1}(G),R^{\times})$
identifies via $\kappa_{G}$ to a group of $R$-valued characters of $G$, hence it
acts on the category $\Rep_{R}(G)$ by twisting the representations. Since this action is
$R$-linear, this induces in turn  an action
of $\Psi_{G}(R)$ by automorphisms of $R$-algebra on $\ZG_{R}(G)$, hence an action on the
set ${\rm Idemp}(\ZG_{R}(G))$ of idempotents of $\ZG_{R}(G)$.

The idempotents of $\ZG_{\CM}(G)$ are known to be
supported on the set of compact elements \cite[Cor. 2.11]{idemp}, hence in particular on
$G^{1}$, so the action of $\Psi_{G}(\CM)$ on  ${\rm Idemp}(\ZG_{\CM}(G))$ factors through an action of
$\Psi^{f}_{G}(\CM)=\Hom(\pi_{1}(G)_{\rm tors},\CM^{\times})=\pi_{0}(\Psi_{G,\CM})$.
Further, the depth $0$ projector $\varepsilon_{0}$ is known to be supported on the set of
topologically unipotent elements \cite[Cor 1.9 b)]{BKV}, hence in particular on $G^{0}$, so $\varepsilon_{0}$ is
invariant under the action of $\Psi_G^{f}(\CM)$.

\begin{Lem}
  \label{lemactionpsif}
  For each $[\tG]\in\TG/\sim$, the associated  idempotent $\varepsilon_{[\tG]}$ in
  $\ZG_{\CM}(G)$ is invariant by $\Psi_G^{f}(\CM)$. Moreover, the primitive idempotents that
  refine $\varepsilon_{[\tG]}$ form a single $\Psi_G^{f}(\CM)$-orbit.
\end{Lem}
\begin{proof}
  Pick $(\sigma,\pi)\in[\tG]$. We know that the direct factor category
  $\Rep^{[\tG]}_{\CM}(G)$ is generated by the projective object
  ${\rm ind}_{G_{\sigma}}^{G}(\pi)$. For any $\psi\in \Psi_{G}(\CM)$, we have
  ${\rm ind}_{G_{\sigma}}^{G}(\pi)\otimes \psi = {\rm
    ind}_{G_{\sigma}}^{G}(\pi\otimes\psi_{|G_{\sigma}})={\rm
    ind}_{G_{\sigma}}^{G}(\pi)$. It follows that $\Rep^{[\tG]}_{\CM}(G)$ is stable under
  the action of $\Psi_{G}(\CM)$, hence $\varepsilon_{[\tG]}$ is invariant.

  Now, we also know that any block of $\Rep^{[\tG]}_{\CM}(G)$ is generated by a
  projective object of the form   ${\rm ind}_{G_{\sigma}^{1}}^{G}(\pi^{1})$
  where $\pi^{1}$ is an irreducible constituent of ${\rm
        ind}_{G_{\sigma}}^{G_{\sigma}^{1}}(\pi)$. But any two such irreducible
  constituents are twists of one another by a character of
  $G_{\sigma}^{1}/G_{\sigma}$. Moreover, the latter group embeds in $\pi_{1}(G)_{\rm tors}$ via
  $\kappa_{G}$, hence the restriction map $\Psi_G^{f}(\CM)\To{}
    \Hom(G_{\sigma}^{1}/G_{\sigma},\CM^{\times})$ is surjective, and the second
  statement of the lemma follows.
\end{proof}

This lemma implies that the only central idempotents of depth $0$ in $\ZG_{\CM}(G)$ that
are invariant by $\Psi_G^{f}(\CM)$ are the $\varepsilon_{T}$ for $T\subset
  \TG/\sim$. Now, observe that the action of $\Psi_{G}(\o\ZM[\frac 1p])$
on ${\rm Idemp}(\ZG_{\o\ZM[\frac 1p]}(G))$ factors over
$\Psi^{f}_{G}(\o\ZM[\frac 1p])$, which is equal to
$\Psi^{f}_{G}(\CM)$. In other words, $\Psi^{f}_{G}(\CM)$ preserves the
subset of idempotents of $\ZG_{\CM}(G)$ that belong to $\ZG_{\o\ZM[\frac 1p]}(G)$.
Therefore, we can restate Corollary \ref{coridemp0} as follows :

\begin{Cor}\label{corpsiinvariant}
  The only $\Psi_G^{f}(\o\ZM[\frac 1p])$-invariant idempotent of $\varepsilon_{0}\ZG_{\o\ZM[\frac 1p]}(G)$
  is $\varepsilon_{0}$. Hence $\Psi_G^{f}(\o\ZM[\frac 1p])$ acts transitively
  on the set of primitive idempotents of $\varepsilon_{0}\ZG_{\o\ZM[\frac 1p]}(G)$.
\end{Cor}

In order to better understand the action of $\Psi_G^{f}(\o\ZM[\frac 1p])$ on ${\rm
      Idemp}(\varepsilon_{0}\ZG_{\o\ZM[\frac 1p]}(G))$, write this group as a product
$\Psi_G^{f}(\o\ZM[\frac 1p])=\Psi_G^{f}(\o\ZM[\frac 1p])_{p}\times\Psi_G^{f}(\o\ZM[\frac 1p])_{p'}$ of a $p$-group and
a $p'$-group.

\begin{Lem} \label{lemprimetop}
  Any idempotent of $\ZG_{\o\ZM[\frac 1p]}(G)$ is invariant under $\Psi_G^{f}(\o\ZM[\frac 1p])_{p'}$.
\end{Lem}
\begin{proof}
  Let $\varepsilon$ be an idempotent of $\ZG_{\o\ZM[\frac 1p]}(G)$ and assume, without
  loss of generality, that it is primitive. On the other hand, let $\psi\in
    \Psi_G^{f}(\o\ZM[\frac 1p])_{p'}$ and assume, without loss of generality that $\psi$ has order
  a power of some prime $\ell\neq p$. In order to prove that
  $\psi\cdot\varepsilon=\varepsilon$, it suffices to find a non-zero object $(V,\pi)$ in
  $\varepsilon\Rep_{\o\ZM[\frac 1p]}(G)$ such that $\pi\otimes\psi\simeq \pi$.  Since the
  composition of $\psi : \pi_{1}(G)_{\rm tors}\To{}\o\ZM[\frac 1p]^{\times}$ with any
  morphism $\o\ZM[\frac 1p] \To{}\o\FM_{\ell}$ is trivial, it suffices to find a non-zero
  object $(V,\pi)$ in $\varepsilon\Rep_{\o\ZM[\frac 1p]}(G)$
  whose $\o\ZM[\frac 1p]$-module structure factors
  over a morphism $\o\ZM[\frac 1p]\To{}\o\FM_{\ell}$.  But $\varepsilon\Rep_{\o\ZM[\frac
        1p]}(G)$ certainly  contains a
  representation of the form
  $\varepsilon.{\rm ind}_{H}^{G}(\o\ZM[\frac 1p])$ for some open
  pro-$p$-subgroup $H$ of $G$. Such a representation being
  projective as a $\o\ZM[\frac 1p]$-module, its reduction modulo any maximal ideal
  containing $\ell$ is non-zero and satisfies the desired property.
\end{proof}

In the next statement, we say that a torus $\Tf$ defined over $F$ is
\emph{$P_{F}$-induced} if  the action of the wild inertia subgroup
$P_{F}$ permutes a basis of $X_{*}(\Tf)$.

\begin{Cor}\label{cornqs}
  Suppose that $p$ does not divide $ |\pi_{1}(\Gf_{\rm der})|$ and that the torus $\Gf_{\rm ab}$ is
  $P_{F}$-induced. Then $\varepsilon_{0}$
  is a primitive idempotent of $\ZG_{\o\ZM[\frac 1p]}(G)$.
\end{Cor}
\begin{proof}
  From the
  inclusions $X_{*}(\Tf_{\rm sc})\subset X_{*}(\Tf_{\rm der})\subset X_{*}(\Tf)$ and the
  isomorphism $X_{*}(\Tf)/X_{*}(\Tf_{\rm der})\simto X_{*}(\Gf_{\rm ab})$,
  we get an exact sequence
  $\pi_{1}(\Gf_{\rm der})\injo{}\pi_{1}(\Gf)\twoheadrightarrow X_{*}(\Gf_{\rm ab}).$
  Applying $I_{F}$-coinvariants, we get an exact sequence
  $\pi_{1}(\Gf_{\rm der})_{I_{F}}\To{}\pi_{1}(\Gf)_{I_{F}}\twoheadrightarrow X_{*}(\Gf_{\rm ab})_{I_{F}}.$
  Since $\pi_{1}(\Gf_{\rm der})$ is finite, this sequence remains exact on the torsion
  subgroups and the $p$-torsion subgroups, so we get an exact sequence
  $$\pi_{1}(\Gf_{\rm der})_{I_{F}, p-\rm tors}\To{}\pi_{1}(\Gf)_{I_{F}, p-\rm
    tors}\twoheadrightarrow X_{*}(\Gf_{\rm ab})_{I_{F}, p-\rm tors}.$$
  Our first assumption implies that $\pi_{1}(\Gf_{\rm der})_{I_{F}, p-\rm tors}=0$, and our
  second assumption implies that
  $X_{*}(\Gf_{\rm ab})_{I_{F}, p-\rm tors}= (X_{*}(\Gf_{\rm ab})_{P_{F}, \rm
    tors})_{I_{F}}=0$. Here the first equality is because $I_{F}$ acts on
  $P_{F}$-coinvariants through a
  finite quotient of order \emph{prime to} $p$, and the vanishing claim is because
  $X_{*}(\Gf_{\rm ab})_{P_{F}}$ is torsion free since $\Gf_{\rm ab}$ is $P_{F}$-induced.
  Therefore $\pi_{1}(G)_{p-\rm tors}=\{0\}$, thus
  $\Psi_G^{f}(\o\ZM[\frac 1p])_{p}=\{1\}$ and we conclude thanks to Lemma
  \ref{lemprimetop} and Corollary \ref{corpsiinvariant}.
\end{proof}

The following example shows that the condition on $\pi_{1}(\Gf_{\rm der})$ is not
always nece\-ssary for $\varepsilon_{0}$ to be a primitive idempotent.

\begin{Exa} \label{Example_PGL_p}
  Let $\Gf={\rm PGL}_{p}$ with split $F$-structure, so that $\pi_{1}(G)=\ZM/p\ZM$ and
  $\kappa_{G} : G\To{} \pi_{1}(G)$ is induced by the valuation of the determinant. Then any
  irreducible supercuspidal representation of the form $\pi={\rm ind}_{{\rm
        PGL}_{p}(O_{F})}^{G}(\bar\pi)$ is invariant under torsion by the group
  $\Psi_{G}^{f}(\CM)$ of characters of $\pi_{1}(G)$, hence so is the only primitive
  idempotent $\varepsilon$ of $\varepsilon_{0}\ZG_{\o\ZM[\frac 1p]}(G)$ such that
  $\varepsilon\pi\neq 0$, showing that $\varepsilon_{0}=\varepsilon$ is primitive by Corollary~\ref{corpsiinvariant}.
\end{Exa}

The next paragraph generalizes the above observation about ${\rm PGL}_{p}$.

\subsection{Special points}
Let $\mathbf{S}$ be a maximal split torus in $\Gf$ and let $\mathbf{Z}$ be the centralizer
of $\mathbf{S}$ in $\Gf$. This is a Levi component of a minimal
$F$-rational parabolic subgroup of $\Gf$. By \cite[Lemma 11.5.6]{kaletha_prasad}, we
know that the canonical map $\pi_{1}(Z)\To{}\pi_{1}(G)$  is injective on torsion
subgroups. Correspondingly,
the map $\Psi_{G}^{f}(\o\ZM[\frac 1p])\To{} \Psi_{Z}^{f}(\o\ZM[\frac 1p])$ is
surjective. Note that the quotient of $\Psi_{G}^{f}(\o\ZM[\frac 1p])$ thus obtained is independent
of the choice of $\mathbf{S}$ since all maximal split tori are $G$-conjugate.

\begin{Lem} \label{lemmeZ}
  The action of $\Psi_{G}^{f}(\o\ZM[\frac 1p])$ on the set of idempotents of
  $\varepsilon_{0}\ZG_{\o\ZM[\frac 1p]}(G)$ factors over the quotient $\Psi_{Z}^{f}(\o\ZM[\frac 1p])$.
\end{Lem}
\begin{proof}
  Let $\mathbf{S}$ and $\mathbf{Z}$ be as above, and pick a special vertex $x$ in the
  apartment corresponding to $\mathbf{S}$ in $\BT$.  By
  \cite[Prop. 7.7.5]{kaletha_prasad} (which reconciles our definition
  of $G^{1}_{\sigma}$ with the one they introduce in the beginning of their section 7.7),    we have $G^{1}_{x}=G_{x} Z^{1}$, hence
  $G^{1}_{x}/G_{x}= Z^{1}/Z^{0}= \pi_{1}(Z)_{\rm tors}$.
  Therefore,
  if we pick a supercuspidal representation $\pi$ of $\o G_{x}$, and an irreducible
  subquotient $\pi^{1}$ of ${\rm ind}_{G_{x}}^{G_{x}^{1}}(\pi)$, then the
  corresponding primitive central idempotent $\varepsilon_{(x,\pi^{1})}\in\ZG_{\CM}(G)$
  is invariant under the kernel of $\Psi_{G}^{f}(\CM)\To{}\Psi_{Z}^{f}(\CM)$.  So let
  $\varepsilon$ be the unique primitive
  idempotent of $\ZG_{\o\ZM[\frac 1p]}(G)$ such that
  $\varepsilon.\varepsilon_{(x,\pi^{1})} = \varepsilon_{(x,\pi^{1})}$. By
  uniqueness, $\varepsilon$ is also invariant under the kernel of
  $\Psi_{G}^{f}(\o\ZM[\frac 1p])\To{}\Psi_{Z}^{f}(\o\ZM[\frac
      1p])$. Since $\Psi_{G}^{f}(\o\ZM[\frac 1p])$ acts transitively
  on the set of primitive idempotents of $\varepsilon_{0}\ZG_{\o\ZM[\frac 1p]}(G)$, we
  conclude that this action factors over $\Psi_{Z}^{f}(\o\ZM[\frac 1p])$.
\end{proof}

\subsection{The quasi-split case : group side} In this subsection, we assume that $\mathbf{G}$ is
quasi-split over $F$. In this case, the centralizer $\mathbf{Z}$ of a maximal split torus
$\mathbf{S}$ is itself a torus, that we denote by $\mathbf{T}:=\mathbf{Z}$. According to
Lemma \ref{lemmeZ}, Lemma \ref{lemprimetop} and Corollary \ref{corpsiinvariant}, the
natural action of $\Psi_{G}^{f}(\o\ZM[\frac 1p])=\Hom(\pi_{1}(G)_{\rm
      tors},\o\ZM[\frac 1p]^{\times})$ induces a transitive action of
$\Psi_{T}^{f}(\o\ZM[\frac 1p])_{p}=\Hom(\pi_{1}(T)_{p-\rm
      tors},\o\ZM[\frac 1p]^{\times})$
on the set of primitive idempotents of $\varepsilon_{0}\ZG_{\o\ZM[\frac 1p]}(G)$.
Therefore, if $\Tf$ is
$P_{F}$-induced,  we have $\pi_{1}(T)_{p-\rm tors}=1$, so
it follows that $\varepsilon_{0}$ is a primitive idempotent in
$\ZG_{\o\ZM[\frac 1p]}(G)$. In particular we have proven the following result.

\begin{The}\label{quasisplittame}
  Suppose that $\Gf$ is quasi-split and tamely ramified over $F$. Then $\varepsilon_{0}$
  is a primitive idempotent of $\ZG_{\o\ZM[\frac 1p]}(G)$.
\end{The}

This theorem  mirrors the fact that the space of tamely ramified Langlands parameters for
$G$ is connected over $\o\ZM[\frac 1p]$, under the same hypothesis, as proved in \cite[Theorem
  4.29]{DHKM}.
Below we will prove more generally that for any quasi-split $\mathbf{G}$, there is a
natural bijection between connected components of the space of tamely ramified Langlands
parameters for $G$ and the set of primitive idempotents in
$\varepsilon_{0}\ZG_{\o\ZM[\frac 1p]}(G)$. On the $\Gf$-side, the main result is the
following one.

\begin{The}  \label{thm_main}
  Suppose $\mathbf{G}$ is quasi-split. Then the action of $\Psi_{T}^{f}(\o\ZM[\frac 1p])_{p}$ on the
  set of primitive idempotents of $\varepsilon_{0}\ZG_{\o\ZM[\frac 1p]}(G)$ is simply transitive.
\end{The}
\begin{proof}
  Let $G':=\kappa_{G}^{-1}(\pi_{1}(T)_{p-{\rm tors}})$ be the inverse image
  in $G$ of $\pi_{1}(T)_{p-\rm tors}$ by
  $\kappa_{G}$ (recall from above that the map $\pi_{1}(T)\To{}\pi_{1}(G)$ is injective on
  torsion subgroups). As already mentioned, it does not depend on the choice of $\mathbf{S}$.
  For any facet $\sigma$ in $\BT_{\bullet}$, we denote by $G_{\sigma}'$ the
  pointwise stabilizer of $\sigma$ in $G'$, so that we have $G_{\sigma}':=G_{\sigma}^{1}\cap G'$. If
  $\sigma$ belongs to the apartment associated to  $\mathbf{S}$, then
  $G_{\sigma}'=G_{\sigma}T'$ where
  $T'=T\cap G'=\kappa_{T}^{-1}(\pi_{1}(T)_{p-\rm tors})$.
  Since $G_{\sigma}\cap T=T^{0}$, we have a short exact sequence
  \begin{equation}
    \bar G_{\sigma} =G_{\sigma}/G_{\sigma}^{+}\injo
    G_{\sigma}'/G_{\sigma}^{+}\twoheadrightarrow \pi_{1}(T)_{p-\rm tors}.\label{eq:1}
  \end{equation}
  We claim that this sequence splits canonically and, more precisely, that there is  a
  canonical decomposition
  \begin{equation}
    G_{\sigma}'/G_{\sigma}^{+} = \left(G_{\sigma}/G_{\sigma}^{+}\right)\times
    \pi_{1}(T)_{p-\rm tors}.\label{eq:2}
  \end{equation}
  To see this,
  recall that there are canonical smooth $\OC_{F}$-models
  $\mathbf{G}_{\sigma}\subset\mathbf{G}'_{\sigma}$ of $\mathbf{G}$ such that
  \begin{enumerate}
    \item $\mathbf{G}_{\sigma}(\OC_{F})=G_{\sigma}$ and
          $\mathbf{G}'_{\sigma}(\OC_{F})=G'_{\sigma}$ and
          $(\mathbf{G}_{\sigma})_{\o\FM_{q}}=((\mathbf{G}'_{\sigma})_{\o\FM_{q}})^{\circ}$.
    \item
          $\mathbf{G}'_{\sigma}$ contains the canonical model $\mathbf{T}'$ of $\mathbf{T}$
          such that $\mathbf{T}'(\OC_{F})=T'$, and
          $\pi_{0}(\mathbf{T'}_{\o\FM_{q}})\simto \pi_{0}((\mathbf{G}'_{\sigma})_{\o\FM_{q}})$,
          while we also have $\pi_{0}(\mathbf{T'}_{\o\FM_{q}})\simto  \pi_{1}(\mathbf{T})_{I_{F},p-\rm
            tors}.$
    \item Denote by  $\o{\mathbf{G}'_{\sigma}}$ the quotient
          of the special fiber $(\mathbf{G}'_{\sigma})_{\FM_{q}}$ of $\mathbf{G}'_{\sigma}$ by
          its unipotent radical. Then
          the short exact sequence (\ref{eq:1}) is obtained by taking the $\FM_{q}$-rational points
          of the sequence
          $$\o{\mathbf{G}_{\sigma}}
            \injo \o{\mathbf{G}'_{\sigma}}
            \twoheadrightarrow \pi_{0}((\mathbf{G}'_{\sigma})_{\o\FM_{q}})=\pi_{1}(\mathbf{T})_{I_{F},p-\rm
            tors}.$$
    \item Denote by $\o{\mathbf{T}'}$ the quotient of the special fiber of $\mathbf{T}'$ by
          its unipotent radical. Then $\mathbf{T}'\injo \mathbf{G}'_{\sigma}$ induces a closed
          immersion $\o{\mathbf{T}'}\injo \o{\mathbf{G}'_{\sigma}}$ and
          $\o{\mathbf{T}'}\cap \o{\mathbf{G}_{\sigma}}=\o{\mathbf{T}'}^{\circ}$ is a maximal
          torus of $\o{\mathbf{G}_{\sigma}}$ and we have an exact sequence
          $$\o{\mathbf{T}'}^{\circ}=\o{\mathbf{T}'}\cap \o{\mathbf{G}_{\sigma}}
            \injo \o{\mathbf{T}'}
            \twoheadrightarrow \pi_{0}((\mathbf{T}')_{\o\FM_{q}})=\pi_{1}(\mathbf{T})_{I_{F},p-\rm
            tors}.$$
  \end{enumerate}

  Here, $\mathbf{G}_{\sigma}$ is the model that would be denoted by
  $\GC_{\sigma}^{0}$ in the notation of \cite[\S 8.3]{kaletha_prasad}, while $\mathbf{G}'_{\sigma}$ is a variant of
  the model denoted by $\GC_{\sigma}^{b}$ there (the latter would correspond
  to $\kappa_{G}^{-1}(\pi_{1}(T)_{\rm tors})$ rather than
  $\kappa_{G}^{-1}(\pi_{1}(T)_{p-\rm tors})$), whose existence follows
  from  Proposition A.5.23 (3) of \emph{loc. cit.}  Then (2) follows from
  (p-primary variants of) Corollaries 11.1.6 and 11.2.1 there. Items (3) and (4) follow
  from the constructions and Corollary 11.7.2 of \emph{loc. cit.}

  Now, since $\o{\mathbf{T}'}^{\circ}(\o\FM_{p})$ is a $p'$-torsion abelian group,
  $H^{1}(\pi_{1}(\mathbf{T})_{p-\rm tors},\o{\mathbf{T}'}^{\circ}(\o\FM_{p}))=\{1\}$  so
  there exists a splitting $\iota :\, \pi_{1}(\mathbf{T})_{I_{F},p-\rm tors}\injo
    \o{\mathbf{T}'}(\o\FM_{p})$ of the last exact sequence. This $\iota$ also provides  a
  splitting $\pi_{1}(\mathbf{T})_{I_{F},p-\rm tors}\injo   \o{\mathbf{G}'_{\sigma}}(\o\FM_{p})$
  of the short exact sequence in item 3 above.
  But since $\mathbf{T}'$    is an abelian group scheme, we see that the conjugation
  action of $\pi_{1}(\mathbf{T})_{I_{F},p-\rm tors}$ on $\o{\mathbf{G}_{\sigma}}$ through
  $\iota$  fixes pointwise the maximal torus $\o{\mathbf{T}'}^{\circ}$ of
  $\o{\mathbf{G}_{\sigma}}$. It follows that this action is inner and, more precisely,
  given by a morphism from $\pi_{1}(\mathbf{T})_{I_{F},p-\rm tors}$ to the image of
  $\o{\mathbf{T}'}^{\circ}$ in the adjoint group of $\o{\mathbf{G}_{\sigma}}$. But such a
  morphism has to be trivial since $\pi_{1}(\mathbf{T})_{I_{F},p-\rm tors}$ is a $p$-group.
  Hence the action of $\pi_{1}(\mathbf{T})_{I_{F},p-\rm tors}$ through $\iota$ is
  trivial on $\o{\mathbf{G}_{\sigma}}$ and we get a decomposition
  $\o{\mathbf{G}'_{\sigma}}=\o{\mathbf{G}_{\sigma}}\times \pi_{1}(\mathbf{T})_{I_{F},p-\rm
    tors}$. Moreover, such a decomposition is unique because
  $\Hom(\pi_{1}(\mathbf{T})_{I_{F},p-\rm tors},
    Z(\o{\mathbf{G}_{\sigma}}(\o\FM_{p})))=\{1\}$.
  Taking $\FM_{q}$-rational points, we get the claimed canonical
  decomposition (\ref{eq:2}).

  Now, this decomposition (\ref{eq:2}) implies that the pro-$p$-radical ${G'_{\sigma}}^{+}$  of
  $G'_{\sigma}$ surjects onto $\pi_{1}(T)_{p-\rm tors}$.
  For any character $\psi$ of $\pi_{1}(T)_{p-\rm tors}$, we therefore get a central idempotent
  $e_{\sigma}^{\psi} \in \HC_{\o\ZM[\frac 1p]}(G'_{\sigma})$ supported on
  ${G'_{\sigma}}^{+}$, and we have $e_{\sigma}^{+}=\sum_{\psi}e_{\sigma}^{\psi}$.
  Again, these idempotents do not depend on the choice of apartment containing $\sigma$,
  since they are given by the restriction of a global character of $G'$ to
  $(G'_{\sigma})^{+}$. In particular, they are invariant under the action of $G$, in the
  sense that $e_{g\sigma}^{\psi}=ge_{\sigma}^{\psi}g^{-1}$ for all $g\in G$. Moreover, if
  $x$ is a vertex of the facet $\sigma$, the pro-$p$-radical ${G'_{x}}^{+}$ of $G'_{x}$
  is a normal subgroup of ${G'_{\sigma}}^{+}$ and we have
  ${G'_{\sigma}}^{+}={G'_{x}}^{+}G_{\sigma}^{+}$. In terms of idempotents, it follows
  that $e_{\sigma}^{\psi}=e_{\sigma}^{+}e_{x}^{\psi}$ for all $\psi$. But then, the proof
  of \cite[Prop. 1.0.6]{lanard} shows that the system of idempotents
  $(e_{\sigma}^{\psi})_{\sigma\in \BT_{\bullet}}$ is consistent in the sense of
  \cite[Def. 2.1]{meyer_resolutions_2010}
  (note that in \cite[Prop. 1.0.6]{lanard} the idempotents are assumed to be supported
  on the parahoric subgroups while here we allow support on a slightly bigger subgroup,
  but this is harmless for the argument there). Then,
  \cite[Thm. 3.1]{meyer_resolutions_2010} tells us that the full subcategories
  $\Rep^{\psi}_{\o\ZM[\frac 1p]}(G) := \{V\in\Rep_{\o\ZM[\frac
        1p]}(G),\, V=\sum_{x\in \BT_{0}} e_{x}^{\psi}V\}$ are Serre subcategories of
  $\Rep^{0}_{\o\ZM[\frac 1p]}(G)$. Since for all $\sigma,\psi$ and $\psi'$ we have
  $\psi\neq \psi'\Rightarrow e_{\sigma}^{\psi}e_{\sigma}^{\psi'}=0$, these categories are
  pairwise orthogonal. Moreover, since for all $\sigma$ we have
  $e_{\sigma}^{+}=\sum_{\psi}e_{\sigma}^{\psi}$, we actually get a decomposition
  $\Rep^{0}_{\o\ZM[\frac 1p]}(G)=\prod_{\psi}\Rep^{\psi}_{\o\ZM[\frac
      1p]}(G)$. Correspondingly, we get a decomposition of $\varepsilon_{0}$ as a sum of
  pairwise orthogonal idempotents
  $\varepsilon_{0}=\sum_{\psi}\varepsilon^{\psi}_{0}$   in $\ZG_{\o\ZM[\frac
        1p]}(G)$. Finally, identifying $\pi_{0}(\Psi_{T})_{p}$ to the group of characters of
  $\pi_{1}(T)_{p-\rm tors}$, our constructions make it clear that the action of
  $\pi_{0}(\Psi_{T})_{p}$ is given by
  $\psi\cdot\varepsilon_{0}^{\psi'}=\varepsilon_{0}^{\psi\psi'}$.
  Since we already know that the action of $\pi_{0}(\Psi_{T})_{p}$ is transitive on
  primitive idempotents, we conclude that each $\varepsilon_{0}^{\psi}$ has to be primitive,
  and that this action is simply transitive.
\end{proof}

\begin{Rem} \label{rempizero}
  Before turning to the dual side, we give an interpretation of the group
  $\Psi_{G}^{f}(\o\ZM[\frac 1p])_{p}$ (through which the natural action of
  $\Psi_{G}(\o\ZM[\frac 1p])$ on the set of
  idempotents in $\varepsilon_{0}\ZG_{\o\ZM[\frac 1p]}(G)$ factors) in terms of
  the group $\pi_{0}(\Psi_{G})$ of connected
  components of $\Psi_{G}$. Indeed, more generally, for any diagonalizable group scheme
  $A=D(M)$ over $\o\ZM[\frac 1p]$, we have an exact sequence $A^{\circ}\injo
    A\twoheadrightarrow \pi_{0}(A)$ where   $A^{\circ}= D(M/M_{p-\rm
    tors})$ is the ``maximal'' connected diagonalizable subgroup scheme of $A$ and
  $\pi_{0}(A)= D(M_{p-\rm tors})$ is a (constant) finite \'etale
  diagonalizable group scheme. In the case $A=\Psi_{G}$, we thus see that
  $\pi_{0}(\Psi_{G})= D(\pi_{1}(G)_{p-\rm tors})$ is the finite constant group scheme associated to the abstract
  group $\Psi_{G}^{f}(\o\ZM[\frac 1p])_{p}=\Hom(\pi_{1}(G)_{p-\rm tors},\o\ZM[\frac
        1p]^{\times})$, and we shall abuse a bit notation by writing
  $$\pi_{0}(\Psi_{G})=\Psi_{G}^{f}(\o\ZM[\hbox{$\frac 1p$}])_{p}.$$
\end{Rem}

\subsection{The quasi-split case: dual group side}

We now explain how the description of primitive idempotents in the last
theorem matches the parametrization of connected components of the
space of tamely ramified Langlands parameters for $G$.  We will use
the definitions and notation from \cite{DHKM}.
Let us denote by $\HG$ ``the'' dual  group of
$\mathbf{G}$, considered as a split reductive group scheme over
$\o\ZM[\frac 1p]$.  Pick a pinning $\rho:=(\HT,\HB, X=\sum_{\alpha\in\Delta^{\vee}}X_{\alpha})$
of $\HG$, whose underlying Borel pair is dual to a Borel pair
$(\mathbf{T},\mathbf{B})$ where $\mathbf{T}$ is as above (a maximally
split maximal torus) and $\mathbf{B}$ is a Borel subgroup of
$\mathbf{G}$ defined over $F$.
Then the $F$-rational structure on $\mathbf{G}$ induces an
action of $W_{F}$ on the root datum of $\HG$, which induces
in turn an action of $W_{F}$ on $\HG$ preserving the
pinning $\rho$.

\begin{Rem} \label{rempsidual}
  Since the morphism $\mathbf{T}_{\rm sc}\To{}\mathbf{T}$ is dual to the morphism
  $\HT\To{}\HT_{\rm ad}$, we see that $\pi_{1}(\mathbf{G})$ is the group of characters of
  the center $Z(\HG)=\ker(\HT\To{}\HT_{\rm ad})$. It follows in particular that
  $$\Psi_{G}=(Z(\hat\Gf)^{I_{F}})_{\fr},$$ as group schemes over $\o\ZM[\frac 1p]$.

  Then, using the last remark, we may slightly abusively identify
  $$ \Psi_{G}^{f}(\o\ZM[\hbox{$\frac 1p$}])_p=\pi_{0}((Z(\HG)^{I_{F}})_{\fr}).$$
\end{Rem}

Let  us now choose a topological generator $\sigma$ of the tame inertia
group $I_{F}/P_{F}$ and denote by $W_{F}^{0}$ the inverse image in
$W_{F}$ of the discrete subgroup of $W_{F}/P_{F}$ generated by
$\sigma$ and Frobenius. According to \cite[\S 1.2]{DHKM}, there is an affine scheme
$ Z^{1}(W_{F}^{0},\HG)_{\rm tame}$ over $\o\ZM[\frac 1p]$ that classifies
$1$-cocycles $W_{F}\To{} \HG$ whose restriction to $P_{F}$ is
\'etale-locally conjugate to the trivial $1$-cocycle $\phi=1_{P_{F}}:P_{F}\To{}\HG$.
This affine scheme carries an action of $\HG$ over $\o\ZM[\frac 1p]$ and
factors as
$$  Z^{1}(W_{F}^{0},\HG)_{\rm tame} = \HG\times^{\HG^{P_{F}}}  Z^{1}(W_{F}^{0},\HG)_{1_{P_{F}}}$$
where $ Z^{1}(W_{F}^{0},\HG)_{1_{P_{F}}}$ is the closed subscheme of
$ Z^{1}(W_{F}^{0},\HG)_{\rm tame}$ where the restriction of
parameters to $P_{F}$ is trivial, and where $\HG^{P_{F}}$ is the
closed subgroup scheme of $\HG$ fixed by $P_{F}$. Note that, in the
notation of \cite{DHKM}, $\HG^{P_{F}}$ would be denoted $C_{\HG}(\phi)$ if $\phi=1_{P_{F}}$.
Since
$ Z^{1}(W_{F}^{0},\HG)_{1_{P_{F}}}=
  Z^{1}(W_{F}^{0}/P_{F},\HG^{P_{F}})$,  we get on quotient stacks
\begin{equation}
  Z^{1}(W_{F}^{0},\HG)_{\rm tame}/\HG =
  Z^{1}(W_{F}^{0}/P_{F},\HG^{P_{F}})/\HG^{P_{F}}.\label{eq:quotients}
\end{equation}
We are interested in parametrizing the connected components of these stacks.
According to Proposition A.13 and Theorem A.12 of \cite{DHKM}, the $\o\ZM[\frac 1p]$-group scheme $\HG^{P_{F}}$ has split reductive neutral
component $\HG^{P_{F},\circ}$ and finite constant
$\pi_{0}(\HG^{P_{F}})$. We are going to prove that the fibers of  the morphism
\begin{equation}
  Z^{1}(W_{F}^{0}/P_{F},\HG^{P_{F}})/\HG^{P_{F}} \To{\mu}
  H^{1}(W_{F}^{0}/P_{F},\pi_{0}(\HG^{P_{F}})),\label{eq:pimorphism}
\end{equation}
(whose target is a finite discrete scheme) are the connected
components of its source.
To this aim, observe that the diagonalizable group scheme
$ Z^{1}(W_{F}^{0}/P_{F},Z(\HG)^{P_{F}})$ acts on
the scheme $ Z^{1}(W_{F}^{0}/P_{F},\HG^{P_{F}})$ by multiplication of
cocycles, and this action is compatible with $\HG^{P_{F}}$-(twisted) conjugation
on $ Z^{1}(W_{F}^{0}/P_{F},\HG^{P_{F}})$. Furthermore, the map $\mu$ is
equivariant if we let $ Z^{1}(W_{F}^{0}/P_{F},Z(\HG)^{P_{F}})$ act on
$H^{1}(W_{F}^{0}/P_{F},\pi_{0}(\HG^{P_{F}}))$ through
$H^{1}(W_{F}^{0}/P_{F},\pi_{0}(Z(\HG)^{P_{F}}))$.

\begin{Lem}
  With the foregoing notation :
  \begin{enumerate}
    \item The natural map $\pi_{0}(\HT^{P_{F}})\To{}\pi_{0}(\HG^{P_{F}})$ is a bijection. In
          particular, $\pi_{0}(\HG^{P_{F}})$ is an abelian $p$-group.
    \item The natural map
          $$\pi_{0}(Z(\HG)^{I_{F}})_{\fr}=H^{1}(\langle\fr\rangle,\pi_{0}(Z(\HG)^{I_{F}}))\To{}
            H^{1}(W_{F}^{0}/P_{F},\pi_{0}(Z(\HG)^{P_{F}}))$$ is an isomorphism.
    \item Similarly, we have an isomorphism $\pi_{0}(\HT^{I_{F}})_{\fr}\simto
            H^{1}(W_{F}^{0}/P_{F},\pi_{0}(\HT^{P_{F}}))$.
    \item  The natural map $H^{1}(W_{F}^{0}/P_{F},\pi_{0}(Z(\HG)^{P_{F}}))
            \To{}H^{1}(W_{F}^{0}/P_{F},\pi_{0}(\HT^{P_{F}}))$ is surjective.
  \end{enumerate}
\end{Lem}
\begin{proof}
  (1) This is Proposition 4.1 d) of \cite{HainesSatake}.

  (2) Recall first that $\pi_{0}(Z(\HG)^{P_{F}})$ is a finite abelian $p$-group, and the
  action of $I_{F}$ on it is through a cyclic $p'$-group. It follows that
  $H^{1}(I_{F}^{0}/P_{F},\pi_{0}(Z(\HG)^{P_{F}}))=\{1\}$ and therefore the map
  $H^{1}(\langle\fr\rangle,\pi_{0}(Z(\HG)^{P_{F}})^{I_{F}})\To{}
    H^{1}(W_{F}^{0}/P_{F},\pi_{0}(Z(\HG)^{P_{F}}))$ is an isomorphism.
  So it remains to see that $\pi_{0}(Z(\HG)^{P_{F}})^{I_{F}}=\pi_{0}(Z(\HG)^{I_{F}})$, which
  follows from the fact that $(X^{*}(Z(\HG))_{P_{F},p-\rm
        tors})_{I_{F}}=X^{*}(Z(\HG))_{I_{F},p-\rm tors}$ since, as above, the action of $I_{F}$
  on $X^{*}(Z(\HG))_{P_{F}}$ is through a cyclic $p'$-group.

  (3) Apply (2) to $\HT$ instead of $\HG$.

  (4) By (2) and (3),
  it suffices to prove
  surjectivity of $\pi_{0}(Z(\HG)^{I_{F}})\To{}\pi_{0}(\HT^{I_{F}})$, i.e. injectivity of
  $X^{*}(\HT)_{I_{F},p-\rm tors}\To{}X^{*}(Z(\HG))_{I_{F},p-\rm tors}$. For this, we start
  from the exact sequence
  $$ X^{*}(\HT_{\rm ad})_{I_{F}}\To{} X^{*}(\HT)_{I_{F}}\To{} X^{*}(Z(\HG))_{I_{F}}\To{} 0$$
  and we observe that, since $X^{*}(\HT_{\rm ad})$ has a basis permuted by $I_{F}$ (given by
  simple roots), its co-invariants $X^{*}(\HT_{\rm ad})_{I_{F}}$ are a free abelian
  group. Moreover, since the above sequence is exact on the left once we tensor it by $\QM$,
  it follows that $X^{*}(\HT_{\rm ad})_{I_{F}}\To{} X^{*}(\HT)_{I_{F}}$ is injective, hence
  $X^{*}(\HT)_{I_{F},\rm tors}\To{}X^{*}(Z(\HG))_{I_{F},\rm tors}$ is
  injective too.
\end{proof}

The following theorem, together with the identification $\Psi^{f}_{T}(\o\ZM[\frac 1p])_p=
  \pi_{0}((\HT^{I_{F}})_{\fr})=\pi_{0}(\HT^{I_{F}})_{\fr}$ of Remark \ref{rempsidual} is the dual companion of Theorem \ref{thm_main}.

\begin{The} \label{thm_main_dual}
  The connected components of $ Z^{1}(W_{F}^{0},\HG)_{\rm tame}/\HG$
  are the fibers of the map $\mu$ of
  (\ref{eq:pimorphism})  through the identification (\ref{eq:quotients}) :
  $$  Z^{1}(W_{F}^{0}/P_{F},\HG^{P_{F}})/\HG^{P_{F}} =
    Z^{1}(W_{F}^{0},\HG)_{\rm tame}/\HG .$$
  Moreover, the action of
  $ Z^{1}(W_{F}^{0}/P_{F},Z(\HG)^{P_{F}})$ on $ Z^{1}(W_{F}^{0}/P_{F},\HG^{P_{F}})$
  induces a simply transitive action of $\pi_{0}(\HT^{I_{F}})_{\fr}$ on
  connected components of  $ Z^{1}(W_{F}^{0},\HG)_{\rm tame}/\HG$.
\end{The}

\begin{proof}
  By construction, the action of
  $ Z^{1}(W_{F}^{0}/P_{F},\pi_{0}(Z(\HG)^{P_{F}}))$ on
  $ Z^{1}(W_{F}^{0}/P_{F},\pi_{0}(\HG^{P_{F}}))$ induces an action of
  $H^{1}(W_{F}^{0}/P_{F},\pi_{0}(Z(\HG)^{P_{F}}))$ on the set of fibers of the map $\mu$.
  By (1) and (4) of the last lemma, the latter action is transitive, and actually factors
  over a simply transitive action of $H^{1}(W_{F}^{0}/P_{F},\pi_{0}(\HT^{P_{F}}))=\pi_{0}(\HT^{I_{F}})_{\fr}$.
  So, to prove the theorem, it suffices to prove that one fiber of $\mu$ is connected.
  The fiber $\mu^{-1}(1)$ of the trivial cohomology class is
  $ Z^{1}(W_{F}^{0}/P_{F},\HG^{P_{F},\circ})/N\HG^{P_{F},\circ}$, where
  $N\HG^{P_{F},\circ}=\{g\in \HG^{P_{F}}, g^{-1}\sigma(g)\in \HG^{P_{F},\circ},
    g^{-1}\fr(g)\in \HG^{P_{F},\circ}\}.$
  So we are left with proving that $Z^{1}(W_{F}^{0}/P_{F},\HG^{P_{F},\circ})$ is connected. In
  order to apply \cite[Thm. 4.29]{DHKM}, we need to show that the action of $W_{F}/I_{F}$
  on $\HG^{P_{F},\circ}$ fixes a pinning. Thanks to  \cite[Prop 4.1.(a)]{HainesSatake},  we know that
  $(\HB^{P_{F},\circ},\HT^{P_{F},\circ})$ is a Borel pair of $\HG^{P_{F},\circ}$ (over
  $\o\ZM[\frac 1p]$) and even that
  $(\HB^{P_{F},\circ},\HT^{P_{F},\circ},X)$ is a pinning of $\HG^{P_{F},\circ}$, \emph{at least
    over $\o\ZM[\frac 1{2p}]$}. Note that this Borel pair and this pinning are clearly stable under $W_{F}/I_{F}$. So,
  when $p=2$, we are done. On the other hand, as explained in the proof of \cite[Prop 4.1]{HainesSatake},  the
  failure for $X$ to provide a pinning of  $\HG^{P_{F},\circ}$ in characteristic $2$ only
  happens when an orbit of simple roots under  $P_{F}$ contains two roots that add up to a
  root. In this case, the orbit must have even order, and this can't happen if $p$ is
  odd. So, in all cases, $(\HB^{P_{F},\circ},\HT^{P_{F},\circ},X)$ is a pinning of
  $\HG^{P_{F},\circ}$ and we get connectedness of $ Z^{1}(W_{F}^{0}/P_{F},\HG^{P_{F},\circ})$ from \cite[Thm. 4.29]{DHKM}.
\end{proof}

\subsection{Proofs of Theorem \ref{the_main} and Corollary \ref{cor_FS}} \label{pf_main}

Finally, Theorem \ref{the_main} follows from Theorems
\ref{thm_main_dual} and \ref{thm_main} through the identification of
Remark \ref{rempsidual}, once we choose base points in the respective
sets of connected components. Any ``natural'' such choice should be
compatible with parabolic induction from the minimal Levi subgroup
$\mathbf{T}$ of $\mathbf{G}$. But the Langlands correspondence for
tori tells us that the principal block of $\Rep_{\o\ZM[\frac
      1p]}^{0}(T)$ (i.e. the one that contains the trivial representation)
should match the principal component of $ Z^{1}(W_{F}^{0},\HT)_{\rm
      tame}$ (i.e. the one that contains the trivial parameter).

As for Corollary \ref{cor_FS}, let $\pi$ be a depth $0$
irreducible representation of $G$ with coefficients in an algebraically closed field
$L$ over $\o\ZM[\frac 1p]$. By Theorem~\ref{thm_main},
there is an element $\psi \in \Psi_{G}^{f}(\o\ZM[\frac 1p])$ that provides a character
$\psi : G\To{}\pi_{1}(G) \To{}\o\ZM[\frac 1p]^{\times}\To{} L^{\times}$ of $G$ that
belongs to the same block of    $\Rep^{0}_{\o\ZM[\frac 1p]}(G)$ as $\pi$. Hence
the Fargues-Scholze parameters $\varphi_{\pi}$ and $\varphi_{\psi}$ are $L$-points of
the same connected component of the $\o\ZM[\frac 1p]$-scheme $Z^{1}(W_{F}^{0},\HG)\sslash\HG$. Since
$Z^{1}(W_{F}^{0},\HG)_{\rm tame}$ is a sum of connected components,  all we
need to do is to check that  $\varphi_{\psi}$
is tamely ramified. By compatibility with parabolic induction \cite[IX.7.2]{FS} this parameter is the
pushforward along
$Z^{1}(W_{F}^{0},\HT)\sslash\HT\To{}Z^{1}(W_{F}^{0},\HG)\sslash\HG$ of the
Fargues-Scholze parameter of the character $\delta_{B}^{\frac 12}.\psi_{|T}$ of $T$,
where $B$ is any Borel subgroup of $G$ with Levi $T$ and $\delta_{B}$ is its modulus
character. Since the Fargues-Scholze correspondence for tori is the usual one \cite[IX.6.4]{FS}, and the
character $\delta_{B}^{\frac 12}.\psi_{|T}$ is trivial
on $T^{0}$,  its parameter is actually trivial on $I_{F}$, see \cite[Thm 1]{Mishra}, hence a fortiori on $P_{F}$.


\subsection{On the general case} We do not know an analogue of Theorem
\ref{thm_main} in the non quasi-split case, but we believe that Theorem
\ref{quasisplittame} holds for all groups that split over a tamely ramified extension, or
more generally, that have a $P_{F}$-induced maximal $F$-torus. The following example is
not accounted for by our different results.

\begin{Exa} \label{Example_PD_p}
  Let $D$ be a division algebra of dimension $p^{2}$ over $F$ and  $\Gf$ the inner form of
  ${\rm PGL}_{p}$ such that $G=D^{\times}/F^{\times}$. Here,  $\pi_{1}(G)=\ZM/p\ZM$ and
  $\kappa_{G} : G\To{} \pi_{1}(G)$ is induced by the valuation of the reduced norm map. Note
  that $G^{0}=O_{D}^{\times}/O_{F}^{\times}$ surjects onto $\FM_{q^{p}}^{\times}/\FM_{q}^{\times}$ and that the
  action of a suitable generator of $G^{1}/G^{0}=\ZM/p\ZM$ on $\FM_{q^{p}}^{\times}$ is via
  the $q^{th}$-power map (relative Frobenius over $\FM_{q}$). Hence, if $\chi$ is any
  character of $\FM_{q^{p}}^{\times}/\FM_{q}^{\times}$ in general position for the action of
  Frobenius, the induced representation $\pi:={\rm
    ind}_{O_{D}^{\times}/O_{F}^{\times}}^{G}\chi$ is irreducible, of depth $0$, and invariant under
  twisting by characters of $G^{1}/G^{0}$. Therefore, the unique  primitive
  idempotent $\varepsilon$ of $\varepsilon_{0}\ZG_{\o\ZM[\frac 1p]}(G)$ such that
  $\varepsilon\pi\neq 0$ is also invariant under such twisting,
  showing that $\varepsilon_{0}=\varepsilon$ is primitive by Corollary~\ref{corpsiinvariant}.
\end{Exa}

\subsection{Deligne-Lusztig parameters} \label{sec:deligne-luszt-param}
Assuming that $\Gf$ is quasi-split and tamely ramified, we recall here the Deligne-Lusztig
invariant $\pi\mapsto s_{\pi}$ used in Theorem \ref{Thm_order_param}. This is a direct
generalization of \cite[\S 3.4]{lanard} to the tamely ramified setting, and with slightly
different notation.

\medskip
\emph{Transfer maps.} Besides Deligne-Lusztig theory itself, the main point is the existence of natural
morphisms of finite schemes :
\begin{equation}
  (\widehat{\o\Gf}_{\sigma}\sslash \widehat{\o\Gf}_{\sigma})^{\Fr=(.)^{q}}
  \To{} (\HG\rtimes \tau\sslash \HG)^{\Fr=(.)^{q}}\label{eq:transfer}
\end{equation}
for all facets $\sigma$ of $\BG$. Here, $\tau$ is a topological generator of tame inertia,
all duals are taken over $\ZM$, and
$\o\Gf_{\sigma}$ denotes the reductive quotient of the special fiber of the Bruhat-Tits $\OC_{F}$-model $\Gf_{\sigma}$ of
$\Gf$ attached to $\sigma$. To define the map (\ref{eq:transfer}), let $\mathbf{S}$ be a maximal split torus
of $\Gf$ whose apartment contains $\sigma$, and write also $\mathbf{S}$ for its
schematic closure in $\Gf_{\sigma}$, which is a maximal split torus therein.
Denote by $\Tf_{\sigma}$ the centralizer of $\mathbf{S}$ in $\Gf_{\sigma}$ and by
$\mathbf{N}_{\sigma}$ its normalizer. These are smooth group schemes over $\OC_{F}$
and the quotient $\mathbf{W}_{\sigma}:=\mathbf{N}_{\sigma}/\mathbf{T}_{\sigma}$ is an \'etale
group scheme.
By our assumptions on $\Gf$, the generic fiber $\Tf$ of $\Tf_{\sigma}$ is a maximal
$F$-torus of $\Gf$ that is tamely ramified and contained in a Borel subgroup defined over $F$, 
while the reductive quotient $\o\Tf_{\sigma}$ of its special fiber identifies to a maximal torus in $\o\Gf_{\sigma}$.
Actually,  $\Tf_{\sigma}$ is the connected Neron model of $\Tf$ over $\OC_{F}$ and
$\o\Tf_{\sigma}$ is the toric part of its special fiber (compare \cite[Prop. 8.2.4]{kaletha_prasad}), so that, taking cocharacters over
geometric points, we have an $\Fr$-equivariant identification
$X_{*}(\Tf_{\sigma})=X_{*}(\Tf)^{I_{F}}=X_{*}(\Tf)^{\tau}$ (cf
\cite[B.7.9]{kaletha_prasad}). Moreover, the latter identification is
equivariant with respect to the specialization map
$\mathbf{W}_{\sigma}(\bar  k_{F})\To{}\mathbf{W}_{\sigma}(\bar F)^{I_{F}}$ (where bars
denote a choice of algebraic closure).
Then,  going to the dual
tori, this means that we have an  $\Fr$-equivariant identification $\widehat{\o\Tf}_{\sigma}=
  (\widehat\Tf)_{I_{F}}=(\widehat\Tf)_{\tau}$
that is equivariant with respect to the map $W_{\widehat{\o\Gf}_{\sigma}}(\widehat{\o\Tf}_{\sigma})
  \injo W_{\HG}(\HT)^{\tau}$ provided by the identifications $W_{\widehat{\o\Gf}_{\sigma}}(\widehat{\o\Tf}_{\sigma})=
  \mathbf{W}_{\sigma}(\bar k_{F})$ and $W_{\HG}(\HT)=\mathbf{W}_{\sigma}(\bar F)$.
Then the desired map (\ref{eq:transfer}) is induced by the following $\Fr$-equivariant composition :
$$ \widehat{\o\Gf}_{\sigma}\sslash \widehat{\o\Gf}_{\sigma} = \widehat{\o\Tf}_{\sigma}\sslash W_{\widehat{\o\Gf}_{\sigma}}(\widehat{\o\Tf}_{\sigma})
  \To{} (\widehat\Tf)_{\tau}\sslash W_{\HG}(\HT)^{\tau}\simeq (\widehat\Tf\rtimes
  \tau)\sslash N_{\HG}(\HT)_{\tau}\To{\sim} (\HG\rtimes\tau)\sslash\HG$$
where $N_{\HG}(\HT)_{\tau}=\{n\in N_{\HG}(\HT), n\tau(n)^{-1}\in\HT\}$.
Here, we have implicitly chosen $\widehat{\o\Tf}_{\sigma}$, resp. $\widehat\Tf$, as the maximal
torus entering the construction of the dual group $\widehat{\o\Gf}_{\sigma}$, resp. $\widehat\Gf$, with its
pinned action by $\Fr$, resp. $\langle\tau,\Fr\rangle$. The fact that the last map is an
isomorphism is a twisted version of the Chevalley-Steinberg theorem, see for example
\cite[Prop. 6.6]{DHKM}.
The above construction of the map
(\ref{eq:transfer}) is independent of the choice of $\mathbf{S}$, since a different choice
would be conjugated under $\Gf_{\sigma}(\OC_{F})$, leading to compatible identifications
between cocharacter groups of the generic and special fibers of the centralizer, as in \cite[Lemme 3.2.1]{lanard}.


\medskip
\emph{Specialization.} The isomorphism $(\widehat\Tf)_{\tau}\sslash W_{\HG}(\HT)^{\tau}\simto
  \HG\rtimes\tau\sslash\HG$ shows that $(\HG\rtimes\tau\sslash\HG)^{\Fr=(.)^{q}}$ is a finite
scheme. More precisely, if $N$ denotes the exponent of the image of
$W_{\HG}(\HT)^{\tau}\rtimes \langle\Fr\rangle$ in ${\rm  Aut}(\HT)_{\tau}$,
we see that the natural morphism
$(\widehat\Tf)_{\tau}[q^{N}-1]\To{}((\widehat\Tf)_{\tau}\sslash
  W_{\HG}(\HT)^{\tau})^{\Fr=(.)^{q}}$ is surjective (where the bracket denotes the kernel of
$t\mapsto t^{q^{N}-1}$). This also allows us to speak of
\emph{the order} of a geometric point of $(\HG\rtimes\tau\sslash\HG)^{\Fr=(.)^{q}}$ as the
order of any lift in $(\widehat\Tf)_{\tau}$.


Now, by finiteness 
we have specialization maps
$$(\HG\rtimes\tau\sslash\HG)^{\Fr=(.)^{q}}(K)\To{}(\HG\rtimes\tau\sslash\HG)^{\Fr=(.)^{q}}(k)$$
for any discrete valuation ring $\Lambda$ with fraction field $K$ and residue field $k$. When $\Lambda$ is
stricly henselian, this map has a section coming from the canonical lifting of roots of
unity from $k$ to $\Lambda$, which induces a section
$(\widehat\Tf)_{\tau}[q^{N}-1](k)\To{}(\widehat\Tf)_{\tau}[q^{N}-1](\Lambda)$.

From this we see that for each  prime $\ell\neq p$ we have a
specialization map
$$(\HG\rtimes\tau\sslash\HG)^{\Fr=(.)^{q}}(\o\QM_{\ell})\To{}(\HG\rtimes\tau\sslash\HG)^{\Fr=(.)^{q}}(\o\FM_{\ell})$$
with a natural section with image  the subset of elements of  \emph{prime-to-$\ell$ order} in
the left hand side. Actually the composition of the section with specialization is the map that takes
a conjugacy class $s$ to its ``prime-to-$\ell$ order'' (or ``$\ell$-regular'') part.
Similarly, for each maximal ideal of $\o\ZM$ containing $\ell$, we have a
specialization map
$$(\HG\rtimes\tau\sslash\HG)^{\Fr=(.)^{q}}(\o\QM)\To{}(\HG\rtimes\tau\sslash\HG)^{\Fr=(.)^{q}}(\o\FM_{\ell})$$
with a section as above.

\medskip
\emph{The Deligne-Lusztig  decomposition of $\Rep_{L}^{0}(G)$.} For $\ell\neq p$,
Deligne-Lusztig theory provides a surjective map
${\rm Irr}_{\o\QM_{\ell}}(\o\Gf_{\sigma}(k_{F}))\To{}
  (\widehat{\o\Gf}_{\sigma}\sslash  \widehat{\o\Gf}_{\sigma})^{\Fr=(.)^{q}}(\o\QM_{\ell})$
whose fibers are called ``Deligne-Lusztig geometric series''.
This map only depends on a choice of
isomorphism $(\o k_{F})^{\times}\simeq (\QM/\ZM)_{p'}$, which we have
already made when picking the topological generator $\tau$ of $I_{F}/P_{F}$.
In particular, it is invariant under automorphisms of $\o\QM_{\ell}$ and descends
to a map  ${\rm Irr}_{\o\QM}(\o\Gf_{\sigma}(k_{F}))\To{}
  (\widehat{\o\Gf}_{\sigma}\sslash  \widehat{\o\Gf}_{\sigma})^{\Fr=(.)^{q}}(\o\QM)$ canonically.
For a $\o\QM$-point $t$ of $(\widehat{\o\Gf}_{\sigma}\sslash \widehat{\o\Gf}_{\sigma})^{\Fr=(.)^{q}}$ we denote by
$e_{\sigma,\o\QM}^{t}$ the central idempotent of $\o\QM[\o\Gf_{\sigma}(k_{F})]$ that
selects the corresponding geometric series. It is known \cite[Thm 9.12]{cabanes_enguehard}
that if $t$ has prime-to-$\ell$ order then $\sum_{t',
    t'_{\ell'}=t}e^{t'}_{\sigma,\o\QM}$  is
$\ell$-integral, in the sense that it belongs to
$\o\ZM_{(\ell)}[\o\Gf_{\sigma}(k_{F})]$.
This allows us to associate a central idempotent
$e_{\sigma,\o\FM_{\ell}}^{t}$ in
$\o\FM_{\ell}[\o\Gf_{\sigma}(k_{F})]$ to any $\o\FM_{\ell}$-point $t$
of $(\widehat{\o\Gf}_{\sigma}\sslash \widehat{\o\Gf}_{\sigma})^{\Fr=(.)^{q}}$ (this is independent
of the choice of a maximal ideal of $\o\ZM_{(\ell)})$.
In this way we get for any algebraically closed field $L$ over $\o\ZM[\frac 1p]$ a
decomposition
$$1=\sum_{t\in (\widehat{\o\Gf}_{\sigma}\sslash
  \widehat{\o\Gf}_{\sigma})^{\Fr=(.)^{q}}(L)} e_{\sigma,L}^{t}$$ of the
unit in $L[\o\Gf_{\sigma}(k_{F})]$ as a sum of pairwise orthogonal central
idempotents.

Now, for
$s\in (\HG\rtimes \tau\sslash\HG)^{\Fr=(.)^{q}}(L)$, we write $e_{\sigma,L}^{s}:=\sum_{t\mapsto s}e_{\sigma,L}^{t}$. Then the
same argument as in \cite[\S 3.4]{lanard} shows that $\sigma\mapsto e_{\sigma,L}^{s}$ is a
$0$-consistent system of idempotents, so that we have a decomposition
$$ \Rep_{L}^{0}(G) = \prod_{s\in (\HG\rtimes \tau\sslash\HG)^{\Fr=(.)^{q}}(L)}
  \Rep_{L}^{s}(G)$$
where $\Rep_{L}^{s}(G):=\left\{V\in \Rep_{L}(G),
  V=\sum_{\sigma\in \BG_{\bullet}}e_{\sigma,L}^{s}V\right\}$.
When $L$ has characteristic $0$, the relation to the decomposition of Section
\ref{sec:unref-depth-texorpdf} is the following.
For  $\tG=(\sigma,\pi_{\sigma})$ a cuspidal pair, denote by $t_{\pi_{\sigma}}$ the $L$-point of $(\widehat{\o\Gf}_{\sigma}\sslash
  \widehat{\o\Gf}_{\sigma})^{\Fr=(.)^{q}}$ corresponding to the geometric Deligne-Lusztig
series of $\pi_{\sigma}$, and by $s$  the image of $t_{\pi_{\sigma}}$ by the map
(\ref{eq:transfer}). Then we have $\Rep^{\tG}_{L}(G)\subset
  \Rep_{L}^{s}(G)$.

\medskip
\emph{Deligne-Lusztig parameters.}
For any irreducible depth $0$ smooth $L$-representation $\pi$, we  define $s_{\pi}$
to be the unique $L$-point of $(\HG\rtimes \tau\sslash\HG)^{\Fr=(.)^{q}}$ such that
$\pi$ is an object of $\Rep_{L}^{s_{\pi}}(G)$. The map $\pi\mapsto
  s_{\pi}$ satisfies the following :
\begin{enumerate}
  \item Suppose $\pi$ is a subquotient of a parabolically induced module
        $i_{P}^{G}(\rho)$ for some $F$-parabolic subgroup
        $\mathbf{P}=\mathbf{M}\mathbf{U}$ of $\Gf$.  Let $\widehat{\mathbf{P}}=\widehat{\mathbf{M}}\widehat{\mathbf{U}}$
        be a $W_{F}$-stable parabolic subgroup of $\HG$  corresponding to $\mathbf{P}$. Then
        $s_{\pi}$ is the image of $s_{\rho}$ by the natural map
        $(\widehat{\mathbf{M}}\rtimes \tau\sslash\widehat{\mathbf{M}})^{\Fr=(.)^{q}}\To{} (\HG\rtimes
          \tau\sslash\HG)^{\Fr=(.)^{q}}$.

  \item Suppose $L=\o\QM_{\ell}$ and $\pi$ is $\ell$-integral (i.e. admits an admissible
        $\o\ZM_{\ell}$-model $L_{\pi}$), and let $\bar\pi$ be any
        irreducible constituent of the reduction of $L_{\pi}$ to $\o\FM_{\ell}$. Then
        $s_{\bar\pi}$ is the image of $s_{\pi}$ by the specialization map $(\HG\rtimes
          \tau\sslash\HG)^{\Fr=(.)^{q}}(\o\QM_{\ell})\To{}(\HG\rtimes
          \tau\sslash\HG)^{\Fr=(.)^{q}}(\o\FM_{\ell})$.
  \item Suppose $L=\o\FM_{\ell}$ and let $\tilde s_{\pi}$ be the image of $s_{\pi}$ by the
        natural  section $(\HG\rtimes
          \tau\sslash\HG)^{\Fr=(.)^{q}}(\o\FM_{\ell})\To{}(\HG\rtimes
          \tau\sslash\HG)^{\Fr=(.)^{q}}(\o\QM_{\ell})$. Then there is an $\ell$-integral
        $\tilde\pi\in {\rm Irr}_{\o\QM_{\ell}}(G)$ with $s_{\tilde\pi}=\tilde s_{\pi}$ and such that $\pi$ is a constituent of the
        reduction of (any stable lattice in) $\tilde\pi$. The same statement holds if we replace
        $\o\QM_{\ell}$ by $\o\QM$ and consider
        specialization and  section maps associated to a choice of a prime ideal $\LC$ of
        $\o\ZM$ containing $\ell$.

\end{enumerate}

Property (1) is proved like in \cite[Th\'eor\`eme 4.4.3]{lanard} and property (2)
follows from the constructions. For property (3) we may use (1) and an inductive argument
to reduce to the case where $\pi$ is supercuspidal. It is then a Jordan-Hölder factor of a finite
length $\o\FM_{\ell}$-representation of the form $\pi={\rm  ind}_{G_{x}Z}^{G}\pi_{x}$ for some
supercuspidal $\pi_{x}\in\Irr_{\o\FM_{\ell}}(\bar G_{x})$. Denote by
$t_{\pi_{x}}\in (\widehat{\o\Gf}_{x}\sslash\widehat{\o\Gf}_{x})^{\Fr=(.)^{q}}(\o\FM_{\ell})$ its Deligne-Lusztig geometric series and
$\tilde t_{\pi_{x}}$ its natural lift in $(\widehat{\o\Gf}_{x}\sslash\widehat{\o\Gf}_{x})^{\Fr=(.)^{q}}(\o\QM_{\ell})$.
Then $\pi_{x}$ is certainly a Jordan-Hölder constituent of the reduction to
$\o\FM_{\ell}$ of some  $\tilde\pi_{x}\in\Irr_{\o\QM_{\ell}}(\bar G_{x})$ with
$t_{\tilde\pi_{x}}=\tilde t_{\pi_{x}}$. 
By supercuspidality of $\pi_{x}$,  $\tilde\pi_{x}$ also has to be cuspidal, so that the
induction $\tilde\rho:={\rm  ind}_{G_{x}Z}^{G}\tilde\pi_{x}$  is cuspidal with  finite length
and trivial central character, hence is $\ell$-integral. Since $\pi$ is a Jordan-Hölder
constituent of the reduction of $\tilde\rho$, it is a constituent of the reduction of some
irreducible subquotient $\tilde\pi$ of $\tilde\rho$. By construction, $s_{\tilde\pi}$ is
the image of $t_{\tilde\pi_{x}}$, hence equals $\tilde s_{\pi}$.

\subsection{Proof of Theorem \ref{Thm_order_param}}\label{sec:proof-theor-refthm}

As explained in the introduction, we argue on the semi-simple rank of $\Gf$. In rank $0$,
$\Gf$ is a torus and the result follows from the compatibility of ${\rm FS_{mot}}$ with
local class field theory (which follows from the construction just as for the $\ell$-adic
version ${\rm FS}_{\ell}$ in \cite[Prop. IX.6.5]{FS}).  For general $\Gf$,  the
compatibility of $\pi\mapsto \varphi_{\pi}(\tau)$ and $\pi\mapsto s_{\pi}$
with parabolic induction and dual Levi functoriality allows us to focus on $\pi$
supercuspidal. After maybe twisting by an unramified character of $G$, we may assume that
$\pi$ is defined over the closure of the prime subfield of $L$, which reduces the problem
to the cases $L=\o\FM_{\ell}$ or $L=\o\QM$. But property (3) above reduces the case
$L=\o\FM_{\ell}$ to the characteristic $0$ case, i.e. we may assume $L=\o\QM$.
Finally,
after maybe twisting, we may assume that our supercuspidal $\pi$ has an admissible model over $\o\ZM[\frac 1p]$.

We now argue on the number $n_{\pi}$ of prime divisors of the order of $s_{\pi}$.  Assume
first that $n_{\pi}>1$ and let $\ell$ be such a prime number,  and $(s_{\pi})^{(\ell)}$ the
prime-to-$\ell$ part of $s_{\pi}$. Pick any maximal ideal $\LC$ of $\o\ZM$ containing
$\ell$. By properties (2) and (3) above, there is an $\LC$-integral $\pi'\in\Irr_{\o\QM}(G)$ with
$s_{\pi'}=(s_{\pi})^{(\ell)}$ and whose reduction has a common irreducible constituent
$\bar\pi$ with the reduction of $\pi$. By our induction hypothesis, any prime dividing the order of
$\varphi_{\pi'}(\tau)$ divides that of $s_{\pi}^{(\ell)}$, hence the same is true for any
prime dividing the order of $\varphi_{\bar\pi}(\tau)$. But $\varphi_{\bar\pi}(\tau)$ is
the image of $\varphi_{\pi}(\tau)$ by the specialization map, so any prime dividing
the order of $\varphi_{\pi}(\tau)$ either divides that of $\varphi_{\bar\pi}(\tau)$ or is
$\ell$, hence it divides the order of $s_{\pi}$.

It remains to deal with the case $n_{\pi}=1$, i.e. $\pi$ unipotent.
In this case, Proposition \ref{prol12padic} below provides us with two
prime numbers  $\ell_1\neq\ell_2$ different from $p$ and two \emph{non-cuspidal} unipotent
representations $\pi_1,\pi_2\in\Irr_{\o\QM}(G)$ such that $\pi\sim_{\ell_{i}}\pi_{i}$ for
$i=1,2$.
Recall from Remark \ref{remblocksequences} that this means that for any
embeddings $\o\QM\injo\o\QM_{\ell_{i}}$,
$\pi\otimes_{\o\QM}\o\QM_{\ell_{i}}$ and $\pi_{i}\otimes_{\o\QM}\o\QM_{\ell_{i}}$ belong to the same block of $\Rep_{\o\ZM_{\ell_{i}}}(G).$ 
By  the description of connected components of
$Z^{1}(W_{F}^{0},\HG)_{\o\ZM_{\ell}}$ in \cite[Thm 4.8]{DHKM},
it follows that we have  $(\varphi_{\pi})_{|I_{F}^{\ell_{i}}}=(\varphi_{\pi_{i}})_{|I_{F}^{\ell_{i}}}$.
On the other hand, by our induction hypothesis, $\varphi_{\pi_{i}}(\tau)$
is trivial, i.e. $\varphi_{\pi_{i}}$ is trivial on $I_{F}$.
So we see that $\varphi_{\pi}$ is trivial on both $I_{F}^{\ell_{1}}$ and
$I_{F}^{\ell_{2}}$, hence it is trivial on $I_{F}$ and $\varphi_{\pi}(\tau)=1$.

We are now left to prove  Proposition \ref{prol12padic}, and we start with a finite field analogue:

\begin{Pro}
  \label{prol12FiniteGroups}
  Let us assume that $q \neq 2$. Let $\Gf$ be a finite reductive group with positive
  semisimple rank, and $\pi\in \Irr(\Gf^{\fr})$ a unipotent cuspidal representation of $\Gf$. Then there exists two prime numbers $\ell_1,\ell_2 \neq p$, $\ell_1 \neq \ell_2$ and two non-cuspidal unipotent representations $\pi_1,\pi_2 \in \Irr(\Gf^{\fr})$ such that $\pi \sim_{\ell_1} \pi_1 $ and $ \pi \sim_{\ell_2} \pi_2$.
\end{Pro}

\begin{proof}
  Similarly as in Section \ref{secBlocksFiniteGroups} (the bijection
  $\dl{\Gf_{\rm ad}^{\fr}}{1} \simto
  \dl{\Gf^{\fr}}{1}$ on unipotent representations is compatible with the respective $\ell$-block partitions; and a group of adjoint type is a direct product of restriction of scalars of simple groups), we can restrict ourselves to the case where $\Gf$ is a simple group.

  There are no unipotent cuspidal in type $\Aa_n$. If $\Gf$ has type ${}^{2}\Aa_n$ (resp. $\Bb_n$, resp. $\Cc_n$, resp. $\Dd_n$, resp. ${}^{2}\Dd_n$) there is no cuspidal unipotent unless $n=s(s+1)/2 - 1$ (resp. $n=s(s+1)$, resp. $n=s(s+1)$, resp. $n=s^2$, resp. $n=s^2$) for some $s$ and in this case there is only one.

  If $q$ is odd, by \cite[Thm. 21.14]{cabanes_enguehard}, $\dl{\Gf^{\fr}}{1}$ is included in the principal 2-block, so we can take $\ell_1=2$.

  If $q$ is even, then take for $\ell_1$  a prime divisor of $q+1$. As in the proof of Theorem \ref{theclassic} we can use the combinatorics of Lusztig symbols and $\beta$-sets to classify unipotent characters and compute their $d$-series. When $d=2$, a symbol corresponds to a $2$-cuspidal representation if it is itself a $1$-cocore. In the above list of cuspidal characters, none of them are $1$-cocores, so there are no $2$-cuspidals. This means that they all belong to the same $2$-series as another unipotent character, hence $\ell_1$ suits.

  Similarly, the cuspidal unipotent of ${}^{2}\Aa_n$ (resp. $\Bb_n$, $\Cc_n$,$\Dd_n$ or ${}^{2}\Dd_n$) is in the same 6-series (resp. 4-series) as another unipotent representation, so we can take $\ell_2$ such that $q$ is of order 6 (resp. 4) mod $\ell_2$ by Theorem \ref{theordre}. This $\ell_2$ also suits.

  We are left with the exceptional groups $\Ff_4$, ${}^{3}\Dd_4$, $\Gg_2$, $\Ee_6$,
  ${}^{2}\Ee_6$, $\Ee_7$ and $\Ee_8$. For these groups,
  we use  a weaker
  version of Theorem \ref{thelblock} that works also with bad primes,
  \cite[Thm. A]{Enguehard}. It  states that if two unipotent representations are in the same
  $d$-series then they are in the same block. Then, looking at tables, for each unipotent
  cuspidal representation $\pi$ but two exceptions, we can find
  two distinct integers $d_{1},d_{2}$, different from $2$, and such that $\pi$ is in the same
  $d_{i}$-series as a non-cuspidal representation (see Table \ref{tablecuspexgroups} with the notations of \cite[\S 13.9]{carter}). It then remains to pick any primes
  $\ell_{1},\ell_{2}$ such that $q$ has order $d_{i}$ modulo $\ell_{i}$ thanks to Theorem \ref{theordre}.

  The two exceptions are the representations $\Gg_2[1]$ and $\Gg_2[-1]$ in the notation from \cite[\S 13.9]{carter}.
  The representation $\Gg_2[1]$ (resp. $\Gg_2[-1]$) is in the same 2-series and 3-series (resp. 2-series and 6-series) as the trivial
  representation, so we can at least take $\ell_1$ such that $q$ has order 3 modulo $\ell_1$ (resp. $q$ has order 6 modulo $\ell_1$). If there
  also exists $\ell_2$ such that the order of $q$ modulo $\ell_2$ is 2 then we are done. So
  the only issue is when $q=2^k-1$, for some integer $k \geq 2$. In this case, we set
  $\ell_2=2$ and we conclude with \cite[Thm. A.bis]{Enguehard}, which implies that
  two unipotent representations in the same $d$-series are in the same
  $2$-block, if $d$ is the order of $q$ modulo 4, which is 2 here.
  \begin{center}
    \begin{table}[h]
      \begin{tabular}{|l|l|}
        \hline
        unipotent cuspidal representations                                                                         & $(d_1,d_2)$ \\
        \hline
        $\Ff_4[-1]$, $\Ee_8[-1] $                                                                                  & $(6,8)$     \\
        \hline
        $\Ff_4[-i]$, $\Ff_4[i]$, $\Ee_8[i]$, $\Ee_8[-i] $                                                          & $(4,8)$     \\
        \hline
        $\Ff_4[\theta]$, $\Ff_4[\theta^2]$, $\Ee_6[\theta]$, $\Ee_6[\theta^2]$, $\Ee_8[\theta^2]$, $\Ee_8[\theta]$ & $(3,6)$     \\
        $\Gg_2[\theta]$, $\Gg_2[\theta^2]$, ${}^3\Dd_4[1]$, ${}^2\Ee_6[\theta]$, ${}^2\Ee_6[\theta^2]$             &             \\

        \hline
        $\Ff_4'[1]$, $\Ee_8''[1]$                                                                                  & $(4,6)$     \\
        \hline
        $\Ff_4''[1]$, $\Ee_8'[1]$, ${}^2\Ee_6[1]$                                                                  & $(3,4)$     \\
        \hline
        $\Ee_7[\xi]$, $\Ee_7[-\xi]$                                                                                & $(6,14)$    \\

        \hline
        $\Ee_8[-\theta^2]$, $\Ee_8[-\theta]$                                                                       & $(6,24)$    \\
        \hline
        $\Ee_8[\xi^4]$, $\Ee_8[\xi^3]$, $\Ee_8[\xi^2]$, $\Ee_8[\xi]$                                               & $(5,10)$    \\
        \hline
        $ {}^3\Dd_4[-1]$                                                                                           & $(6,12)$    \\

        \hline
      \end{tabular}
      \label{tablecuspexgroups}
      \caption{Suitable pairs $(d_1,d_2)$ for unipotent cuspidal representations of exceptional groups}
    \end{table}
  \end{center}
\end{proof}

\begin{Rem}
  The assumption $q\neq 2$ is needed here for the simple group ${}^{2}\Aa_2$. For instance the group $U_{3}(\mathbb{F}_2)$ has three unipotent representations, the trivial, the Steinberg and $\sigma$ a cuspidal representation. When $\ell = 3$, these three representations are in the same $3$-block. However, the cardinal of $U_{3}(\mathbb{F}_2)$ is $2^3 3^4$, so every odd prime $\ell \neq 3$ is banal, and $\{ \sigma \} $ is a $\ell$-block.
\end{Rem}

\begin{Pro} \label{prol12padic}
  Let $\Gf$ be a reductive group of positive semisimple rank over $F$, and assume that $k_{F}\neq\FM_{2}$. Let $\pi$ be a
  unipotent supercuspidal $\o\QM$-representation of $\Gf$.
  Then there exist
  prime numbers  $\ell_1\neq\ell_2$ different from $p$ and two \emph{non-cuspidal} unipotent
  representations $\pi_1,\pi_2\in\Irr_{\o\QM}(G)$ such that $\pi\sim_{\ell_{i}}\pi_{i}$ for
  $i=1,2$.
\end{Pro}

\begin{proof}
  Let  $\tG\in\TG$ be an unrefined depth 0 type such that $\pi \in
    \Rep^{[\tG]}_{\o\QM}(G)$. Necessarily, $\tG=[x,\pi_x]$ with $x$ a vertex, since $\pi$ is
  supercuspidal. Then $\o\Gf_{x}$ has positive semisimple rank hence, by Proposition
  \ref{prol12FiniteGroups}, there exist two prime numbers
  $\ell_{1},\ell_2$ and two unipotent non-cuspidal representations $\pi_{x,1}$ and
  $\pi_{x,2}$ of $\o\Gf_{x}(k_{F})$ such that $\pi_x \sim_{\ell_1} \pi_{x,1}$ and $\pi_x
    \sim_{\ell_2} \pi_{x,2}$. Let $i \in \{ 1,2 \}$. We consider the cuspidal support of
  $\pi_{x,i}$ which is of the form $(\o\Gf_{\sigma_i}(k_{F}),\tau_i)$ and provides us with
  an  unrefined depth 0 type $\tG_i \in \TG$. Let  $T$ be a minimal subset of $\TG/\sim$
  containing $\tG$ such that the idempotent $\varepsilon_{T}$ is $\ell_{i}$-integral, i.e.
  belongs to $\ZG_{\o\ZM_{(\ell_i)}}(G)$. As in Proposition \ref{proeTintegral}, we have
  $\tG_i \in T$, hence the summand  $\varepsilon_{T}\Rep_{\o\QM}(G)$ contains non cuspidal irreducible
  representations.
  Although $\varepsilon_{T}$ need not be primitive,  we deduce from Lemma \ref{lemactionpsif}
  that $\Psi_G^{f}(\o\QM)$ acts transitively on the set of primitive idempotents in
  $\ZG_{\o\ZM_{(\ell_i)}}(G)$ that refine $\varepsilon_{T}$.
  It follows that any $\ell_{i}$-block inside $\varepsilon_{T}\Rep_{\o\QM}(G)$ contains non cuspidal irreducible
  representations. This applies in particular to the one that contains $\pi$.
\end{proof}


\bibliographystyle{amsalpha}
\bibliography{biblio}
\end{document}